\documentclass[11pt]{amsart}

\usepackage{amssymb}
\usepackage{url}

\usepackage[utf8]{inputenc}
\usepackage[T1]{fontenc}
\usepackage{wasysym}
\usepackage{amsmath}
\usepackage{tikz-cd}

\usepackage{mathrsfs,multicol}
\usepackage[pagebackref, hypertexnames=false, colorlinks, citecolor=red, linkcolor=red]{hyperref}

	\normalbaselines						  %
\def\RR{\mathbb{R}}

\def\CC{\mathbb{C}}
\def\NN{\mathbb{N}}

\newcommand{\al}{{\alpha}}
\newcommand{\la}{{\lambda}}
\newcommand{\f}{{\varphi}}

\newcommand{\cV}{{\mathcal{V}}}
\newcommand{\cX}{{\mathcal{X}}}

\newcommand{\fdot}{\,\cdot\,}

\makeatletter
\def\Ddots{\mathinner{\mkern1mu\raise\p@
\vbox{\kern7\p@\hbox{.}}\mkern2mu
\raise4\p@\hbox{.}\mkern2mu\raise7\p@\hbox{.}\mkern1mu}}
\makeatother

\newcommand{\cH}{\mathcal{H}}
\newcommand{\cC}{\mathcal{C}}
\newcommand{\cB}{\mathcal{B}}
\newcommand{\cF}{\mathcal{F}}
\newcommand{\cA}{\mathcal{A}}

\newcommand{\cD}{\mathcal{D}}

\newcommand{\cM}{\mathcal{M}}

\DeclareMathOperator{\spa}{span}

\newcommand{\ci}[1]{_{ {}_{\scriptstyle #1}}}
\newcommand{\ti}[1]{_{\scriptstyle \text{\rm #1}}}
\catcode`@=11
\chardef\mathlig@atcode\count255

\def\actively#1#2{\begingroup\uccode`\~=`#2\relax\uppercase{\endgroup#1~}}
\def\mathlig@gobble{\afterassignment\mathlig@next@cmd\let\mathlig@next= }
\def\mathlig@delim{\mathlig@delim}
\def\mathlig@defcs#1{\expandafter\def\csname#1\endcsname}
\def\mathlig@let@cs#1#2{\expandafter\let\expandafter#1\csname#2\endcsname}
\def\mathlig@appendcs#1#2{\expandafter\edef\csname#1\endcsname{\csname#1\endcsname#2}}

\def\mathlig#1#2{\mathlig@checklig#1\mathlig@end\mathlig@defcs{mathlig@back@#1}{#2}\ignorespaces}

\def\mathlig@checklig#1#2\mathlig@end{%
 \expandafter\ifx\csname mathlig@forw@#1\endcsname\relax
 \expandafter\mathchardef\csname mathlig@back@#1\endcsname=\mathcode`#1%
 \mathcode`#1"8000\actively\def#1{\csname mathlig@look@#1\endcsname}%
 \mathlig@dolig#1\mathlig@delim
\fi
\mathlig@checksuffix#1#2\mathlig@end
}

\def\mathlig@checksuffix#1#2\mathlig@end{%
\ifx\mathlig@delim#2\mathlig@delim\relax\else\mathlig@checksuffix@{#1}#2\mathlig@end\fi
}
\def\mathlig@checksuffix@#1#2#3\mathlig@end{%
\expandafter\ifx\csname mathlig@forw@#1#2\endcsname\relax\mathlig@dosuffix{#1}{#2}\fi
\mathlig@checksuffix{#1#2}#3\mathlig@end
}

\def\mathlig@dosuffix#1#2{%
\mathlig@appendcs{mathlig@toks@#1}{#2}%
\mathlig@dolig{#1}{#2}\mathlig@delim
}

\def\mathlig@dolig#1#2\mathlig@delim{%
 \mathlig@defcs{mathlig@look@#1#2}{%
 \mathlig@let@cs\mathlig@next{mathlig@forw@#1#2}\futurelet\mathlig@next@tok\mathlig@next}%
 \mathlig@defcs{mathlig@forw@#1#2}{%
  \mathlig@let@cs\mathlig@next{mathlig@back@#1#2}%
  \mathlig@let@cs\checker{mathlig@chck@#1#2}%
  \mathlig@let@cs\mathligtoks{mathlig@toks@#1#2}%
  \expandafter\ifx\expandafter\mathlig@delim\mathligtoks\mathlig@delim\relax\else
  \expandafter\checker\mathligtoks\mathlig@delim\fi
  \mathlig@next
 }%
 \mathlig@defcs{mathlig@toks@#1#2}{}%
 \mathlig@defcs{mathlig@chck@#1#2}##1##2\mathlig@delim{%
  \ifx\mathlig@next@tok##1%
   \mathlig@let@cs\mathlig@next@cmd{mathlig@look@#1#2##1}\let\mathlig@next\mathlig@gobble
  \fi
  \ifx\mathlig@delim##2\mathlig@delim\relax\else
   \csname mathlig@chck@#1#2\endcsname##2\mathlig@delim
  \fi
 }%
 \ifx\mathlig@delim#2\mathlig@delim\else
  \mathlig@defcs{mathlig@back@#1#2}{\csname mathlig@back@#1\endcsname #2}%
 \fi
}%

\catcode`@\mathlig@atcode

\mathchardef\ordinarycolon\mathcode`\:
\def\vcentcolon{\mathrel{\mathop\ordinarycolon}}
\mathlig{:=}{\vcentcolon=}
\mathlig{::=}{\vcentcolon\vcentcolon=}

\numberwithin{equation}{section}

\theoremstyle{plain}
\newtheorem{theo}{Theorem}[section]
\newtheorem{cor}[theo]{Corollary}
\newtheorem{lem}[theo]{Lemma}
\newtheorem{prop}[theo]{Proposition}

\theoremstyle{definition}
\newtheorem{defn}[theo]{Definition}

\newtheorem*{theorem*}{Theorem}
\theoremstyle{remark}

\newtheorem*{ex*}{Example}
\theoremstyle{remark}
\newtheorem*{exs*}{Examples}
\theoremstyle{remark}
\newtheorem*{rems*}{Remarks}
\theoremstyle{remark}
\newtheorem*{rem*}{Remark}
\newtheorem{rem}[theo]{Remark}
\theoremstyle{definition}
\newtheorem*{idea*}{Idea}

\title[Powers of the Jacobi Differential Operator]{Properties and Decompositions of Domains for Powers of the Jacobi Differential Operator}

\author{Dale~Frymark}
\address{Department of Mathematics, Stockholm University, Kr\"aftriket 6, 106 91 Stockholm, Sweden.}
\email{dale@math.su.se}

\author{Constanze~Liaw}
\address{Department of Mathematical Sciences, University of Delware, 501 Ewing Hall, Newark, DE  19716, USA; and 
CASPER, Baylor University, One Bear Place \#97328,      
 Waco, TX  76798, USA.}
\email{liaw@udel.edu}

\thanks{The work of Constanze Liaw was supported by the National Science Foundation under the grant DMS-1802682.}


\begin{document}

\begin{abstract}

We set out to build a framework for self-adjoint extension theory for powers of the Jacobi differential operator that does not make use of classical deficiency elements. Instead, we rely on simpler functions that capture the impact of these elements on extensions but are defined by boundary asymptotics. This new perspective makes calculations much more accessible and allows for a more nuanced analysis of the associated domains.

The maximal domain for $n$-th composition of the Jacobi operator is characterized in terms of a smoothness condition for each derivative, and the endpoint behavior of functions in the underlying Hilbert space can then be classified, for $j\in\NN_0$, by $(1-x)^j$, $(1+x)^j$, $(1-x)^{-\al+j}$ and $(1+x)^{\beta+j}$. Most of these behaviors can only occur when functions are in the associated minimal domain, and this leads to a formulation of the defect spaces with a convenient basis. Self-adjoint extensions, including the important left-definite domain, are then given in terms of the new basis functions for the defect spaces using GKN theory. Comments are made for the Laguerre operator as well.
\end{abstract}

\maketitle

\setcounter{tocdepth}{1}
\tableofcontents

%%%%%%%%%%%%%%%%%%%%%%%%%%%%%
%%%%%%%%%%%%%%%%%%%%%%%%%%%%%
\section{Introduction}\label{s-Intro}
%%%%%%%%%%%%%%%%%%%%%%%%%%%%%
%%%%%%%%%%%%%%%%%%%%%%%%%%%%%

The study of self-adjoint extensions of Sturm--Liouville operators has a long history (i.e.~\cite{AK, G, W1, W2, W}), but very little is known about the powers of these operators, and much less about more general differential operators. For instance, the abstract format of boundary conditions that yield self-adjoint extensions of general differential operators has only recently been described, see e.g.~ \cite{ASZ1, ASZ2, ASZ3, AZ, AZBook}. Other progress was made towards the study of self-adjoint extensions of powers of differential operators when Littlejohn and Wellman characterized the left-definite domains \cite{LW02} associated with self-adjoint differential operators that are bounded below. Left-definite domains, related to the study of left-definite Sturm--Liouville problems, are self-adjoint extensions that can be characterized via powers of the initial operator, see Subsection \ref{ss-leftdef} for a detailed discussion. The broader study of these spaces dates back to at least 1973, when Pleijel investigated the Legendre polynomials \cite{P1, P2} and gave a description comparable to that of Proposition \ref{t-legendremax}.

In order to compensate for studying higher order equations, we are primarily concerned with the powers of Sturm--Liouville operators that possess a complete set of orthogonal eigenfunctions. The Bochner classification \cite{B} tells us that, up to a complex linear change of variable, the only Sturm--Liouville operators with polynomial eigenfunctions are those of Jacobi, Hermite, Laguerre and Bessel. Of these, the Jacobi differential operator requires the most boundary conditions to be self-adjoint, so our analysis is centered around this operator. The Legendre expression is a special case of the Jacobi expression that has an immense amount of literature, so it is discussed in Subsection \ref{ss-legendreexample}. The new framework built to discuss self-adjoint extensions of these operators should also reduce to cover the Laguerre expression as well (see Remark \ref{r-laguerre}). Hermite and Bessel operators are already essentially self-adjoint.

The left-definite domains associated with powers of the Jacobi operator are particularly useful, as they contain all of the Jacobi polynomials (Theorem \ref{t-leftdefortho}) and offer a starting point for the study of other self-adjoint extensions. Unfortunately, while these left-definite spaces have descriptions, little is known about which boundary conditions should be imposed to yield them from their maximal domains. A previous manuscript of the authors \cite{FFL} gives examples of left-definite domains for low powers of the Legendre operator, but achieves limited results for general powers. This was mostly due to difficulties handling the sesquilinear forms generated by the Green's formula, see equation \eqref{e-njacobisesqui}.

In the current paper, a different approach is taken. The Jacobi polynomials, for $\al,\beta\in[0,1)$, can be decomposed as
\begin{align*}
    P_m^{(\al,\beta)}(x)=\sum_{j=0}^m a_j(1-x)^{m-j}\cdot (1+x)^j,
\end{align*}
and are dense in $L^2_{\al,\beta}(-1,1):=L^2[(-1,1),(1-x)^{\al}(1+x)^{\beta}dx]$. The functions $(1-x)^j$ and $(1+x)^j$, $j\in\NN_0$, describe the behavior of these polynomials near the endpoints. Likewise, the second linear independent solution to the Jacobi differential equation can be described by the functions $(1-x)^{-\al+j}$ and $(1+x)^{-\beta+j}$ near the endpoints; Remark \ref{r-choices} explains this in more detail. These four classes of functions are the key to describing the maximal and minimal domains associated with powers of the Jacobi operator. Likewise, the defect spaces can be written with these functions as a basis. This allows the language of self-adjoint extension theory, Glazman--Krein--Naimark (GKN) theory, to be translated so that all self-adjoint extensions of powers of the Jacobi operator can be described via these functions (see Corollary \ref{t-finallyamatrix}).

GKN theory (see Theorem \ref{t-gknmatrix}) is usually implemented by finding basis elements of the defect spaces and combining them in some way determined by a unitary matrix. These basis elements, called deficiency elements in the literature, are solutions to 
\begin{align*}
    \ell[f](x)=\pm if(x),
\end{align*}
where $\ell[\cdot]$ is the symmetric operator of interest. Modern uses of this method are found, for instance, in \cite{AK, GT}. Unfortunately, these deficiency elements are difficult to identify for singular Sturm--Liouville expressions, and even more difficult to work with. The deficiency elements for the Jacobi expression are Jacobi functions of the first and second kind with complex indices, and therefore computing inner products and sesquilinear forms is practically unfeasible. This inability to handle explicit solutions also prohibits the use of other general results, e.g.~those in \cite{ASZ2, ASZ3}. These obstacles were the main motivation for studying the use of orthogonal polynomials as GKN boundary conditions in \cite{FFL}. The functions determined by endpoint behavior above are found to express the contribution these deficiency elements have on self-adjoint extensions.

The improved framework for self-adjoint extension theory allows, in particular, for improved descriptions of the left-definite domain. Several sets of boundary conditions, both in GKN and non-GKN formats, are discussed and proven to be equal in Section \ref{s-equivalence}. These descriptions generalize and confirm a conjecture from \cite{FFL} (and a related one from \cite{LWOG}) that applies to all left-definite domains of Sturm--Liouville operators with a complete set of orthogonal polynomials. While we do not consider the spectral properties of self-adjoint extensions, forthcoming work explores the impact of the new framework.

% The defect indices of the Legendre and Jacobi differential equations are $(1,1)$. In much of this paper we work with connected boundary conditions, as in \cite{FFL}. Connected means that we equate the complete sesquilinear form to zero. %, like in the Definition \ref{d-min}.
% This is in contrast to some other literature (e.g.~\cite{AK, BEZ, E, LW15, Z}), where separated boundary conditions are used. Separated means that conditions are imposed at each endpoint. While the defect indices do not change, the number of imposed conditions when working with separated condition are twice as many than when working with connected ones. For example, the $n$th power of the Jacobi differential equation has defect indices $(n,n)$. That means we will have $n$ (connected) boundary conditions, whereas there would be $2n$ conditions if we were using separated boundary conditions.
% This is a matter of preference, but using connected boundary conditions simplifies calculations considerably, as they are easier to access via the sesquilinear form dealt with by GKN theory. The Naimark Patching Lemma \cite{N, Z} guarantees that separated GKN boundary conditions can be ``patched'' together to form a connected one.

%%%%%%%%%%%%%%%%%%%%%%%%%%%%%
%%%%%%%%%%%%%%%%%%%%%%%%%%%%%
\subsection{Outline}
%%%%%%%%%%%%%%%%%%%%%%%%%%%%%
%%%%%%%%%%%%%%%%%%%%%%%%%%%%%

Section \ref{s-bg} introduces the background necessary to tackle self-adjoint extensions of higher-order differential equations. Left-definite theory, in Subsection \ref{ss-leftdef}, creates a continuum of Hilbert spaces from a differential operator that are given via composition, and indicates when a self-adjoint extension possesses a complete set of orthogonal eigenfunctions. The boundary conditions for these domains are formulated via GKN theory, and require the introduction of several subspaces: the maximal domain, defect spaces, and minimal domain.

Section \ref{s-mindomains} analyzes the domains of both the Legendre and Jacobi differential operators. In this scenario, decompositions of the maximal and minimal domains form the base case for how our methods will work with compositions. However, we present the $n=1$ (operator that is not composed) case from a different perspective that is crucial to our analysis. The definition of the left-definite operator is easily verified as a byproduct of these domain decompositions. 

Section \ref{s-powers} hosts the main results of the paper. Namely, the maximal domains associated with compositions of the Jacobi operator are each characterized in terms of smoothness conditions that originate from the sesquilinear form. This origin makes calculations within the maximal domain streamlined, and as a result the minimal domain is found to possess all functions of the type $(1-x)^j$, $(1-x)^{-\al+j}$, $(1+x)^j$ and $(1+x)^{-\beta+j}$ for $j\geq n$. The span of the finitely-many leftover functions, when $j\in\NN_0$ and $j<n$, are proven to be equal to the defect spaces in Subsection \ref{ss-nextensions}, forming a basis that is much easier to use than the deficiency elements present in classical self-adjoint extension theory. Remark \ref{r-laguerre} discusses a possible reduction of the method for powers of the Laguerre operator. 

Section \ref{s-equivalence} studies the ramifications of the results in Section \ref{s-powers}  on left-definite theory. A conjecture from \cite{FFL} concerning how left-definite spaces can be written as GKN boundary conditions is  generalized and proven for powers of Jacobi operators. This result allows left-definite spaces to be described in several different forms, with and without GKN boundary conditions, giving further insight into the structure of the domains.

%%%%%%%%%%%%%%%%%%%%%%%%%%%%%
%%%%%%%%%%%%%%%%%%%%%%%%%%%%%
\subsection{Notation}
%%%%%%%%%%%%%%%%%%%%%%%%%%%%%
%%%%%%%%%%%%%%%%%%%%%%%%%%%%%

We use $\ell$ to denote differential expressions (on a separable Hilbert space $\cH$), although we mostly work with general Sturm--Liouville expressions in Lagrangian symmetric form. Sets and spaces are generally denoted with ``mathcal"; the Hilbert space $\cH$, the minimal domain $\cD\ti{min}$, the defect spaces $\cD_+$ and $\cD_-$, etc.

The notation $\{\ell,\cX\}$ will refer to an operator that acts via the expression $\ell$ on the domain $\cX$. Since we work with unbounded operators, they are defined on dense subspaces $\cX\subsetneq \cH$. The maximal domain, the largest subset of $\cH$ with $\ell (\cD\ti{max})\subset \cH$, is denoted by $\cD\ti{max}$,  or $\cD\ti{max}(\ell)$ to emphasize the expression. Boldface letters are used for operators, with the maximal and minimal operators abbreviated as ${\bf L}\ti{max}=\{\ell,\cD\ti{max}\}$ and ${\bf L}\ti{min}=\{\ell,\cD\ti{min}\}$. In this context, ${\bf L}=\{\ell, \cD_{\bf L}\}$ is used to denote self-adjoint operators. The domain of a general operator ${\bf  A}$ is referred to by $\cD({\bf  A})$.

It is a common goal of many results to show that evaluation of a sesquilinear form is finite. These evaluations involve limits, and finite constants are thus often removed from the calculations as they are unimportant to convergence. We use the notation $\approx$, instead of =, to denote this removal of constants when rearranging terms.

%%%%%%%%%%%%%%%%%%%%%%%%%%%%%
%%%%%%%%%%%%%%%%%%%%%%%%%%%%%
\section{Classical Self-Adjoint Extension Theory and Left-Definite Theory}\label{s-bg}
%%%%%%%%%%%%%%%%%%%%%%%%%%%%%
%%%%%%%%%%%%%%%%%%%%%%%%%%%%%

Consider the classical Sturm--Liouville differential equation
\begin{align}\label{d-sturmdif}
\dfrac{d}{dx}\left[p(x)\dfrac{dy}{dx}\right]+q(x)y=-\lambda w(x)y,
\end{align}
where $y$ is a function of the independent variable $x$, $p(x),w(x)>0$ a.e.~on $(a,b)$ and $q(x)$ real-valued a.e.~on $(a,b)$. 
Furthermore, $1/p(x),q(x),w(x)\in L^1\ti{loc}[(a,b),dx]$. Additional details about Sturm--Liouville theory can be found, for instance, in \cite{AG, BEZ, E, GZ, Z}.
This differential expression can be viewed as a linear operator, mapping a function $f$ to the function $\ell[f]$ via
\begin{align}\label{d-sturmop}
\ell[f](x):=-\dfrac{1}{w(x)}\left(\dfrac{d}{dx}\left[p(x)\dfrac{df}{dx}(x)\right]+q(x)f(x)\right).
\end{align}
This unbounded operator acts on the Hilbert space $L^2[(a,b),w]$, endowed with the inner product 
$
\langle f,g\rangle:=\int_a^b f(x)\overline{g(x)}w(x)dx.
$
In this setting, the eigenvalue problem $\ell[f](x)=\lambda f(x)$ can be considered. The expression $\ell[\fdot]$ defined in equation \eqref{d-sturmop} has been well-studied, see \cite{I} for an in-depth discussion of its relation to orthogonal polynomials. However, the operator $\{\ell,L^2[(a,b),w]\}$ is not self-adjoint a priori. Additional boundary conditions are required to ensure this property.

Furthermore, the operator $\ell^n[\fdot]$ is defined as the operator $\ell[\fdot]$ composed with itself $n$ times, creating a differential operator of order $2n$. Every formally symmetric differential expression $\ell^n[\fdot]$ of order $2n$ with coefficients $a_k:(a,b)\to\RR$ and $a_k\in C^k(a,b)$, for $k=0,1,\dots,n$ and $n\in\NN$, has the {\em Lagrangian symmetric form} 
\begin{align}\label{e-lagrangian}
\ell^n[f](x)=\sum_{j=1}^n(-1)^j(a_j(x)f^{(j)}(x))^{(j)}, \text{ } x\in(a,b).
\end{align}
Further details can be found in \cite{DS, ELT, LWOG}.

The classical differential expressions of Jacobi, Laguerre and Hermite are all semi-bounded and admit such a representation. Semi-boundedness is defined as the existence of a constant $k\in\RR$ such that for all $x$ in the domain of the operator ${\bf A}$ the following inequality holds:
$$\langle {\bf A}x,x\rangle\geq k\langle x,x\rangle.$$
This additional property, combined with self-adjointness, allows for a continuum of nested Hilbert spaces to be defined within $L^2[(a,b),w]$ via the expressions $\ell^n[\fdot]$. Indeed, this continuum is a Hilbert scale, and many facts about the spectrum and the operators can be deduced using this point of view (e.g.~\cite{DHS, LW13}). More details about Hilbert scales can be found in \cite{AK, KP}. This particular Hilbert scale with self-adjoint operators that are semi-bounded is the topic of left-definite theory \cite{LW02}.

%%%%%%%%%%%%%%%%%%%%%%%%%%%%%
%%%%%%%%%%%%%%%%%%%%%%%%%%%%%
\subsection{Left-Definite Theory}\label{ss-leftdef}
%%%%%%%%%%%%%%%%%%%%%%%%%%%%%
%%%%%%%%%%%%%%%%%%%%%%%%%%%%%

Left-definite theory deals primarily with the spectral theory of Sturm--Liouville differential operators, while the terminology itself can be traced back to Weyl in 1910 \cite{W}. More explicit connections between classical left-definite Sturm--Liouville problems and boundary conditions date back to at least 1973, when Pleijel studied the Legendre polynomials \cite{P1, P2}. A general framework for the left-definite theory of bounded-below, self-adjoint operators in a Hilbert space wasn't developed until 2002 in the landmark paper by Littlejohn and Wellman \cite{LW02}. Specifically, the left-definite theory allows one to generate a scale of operators (by composition), each of which possess the same spectrum as the original. 

Let $\cV$ be a vector space over $\CC$ with inner product $\langle\fdot,\fdot\rangle$ and norm $||\fdot||$. The resulting inner product space is denoted $(\cV,\langle\fdot,\fdot\rangle)$.

\begin{defn}[{\cite[Theorem 3.1]{LW02}}] \label{t-ldinpro}
Suppose ${\bf A}$ is a self-adjoint operator in the Hilbert space $\cH=(\cV,\langle\fdot,\fdot\rangle )$ that is bounded below by $kI$, where $k>0$. Let $r>0$. Define $\cH_r=(\cV_r, \langle\fdot,\fdot\rangle_r)$ with
$$\cV_r=\cD ({\bf A}^{r/2})$$
and
$$\langle x,y\rangle_r=\langle {\bf A}^{r/2}x,{\bf A}^{r/2}y\rangle \text{ for } (x,y\in \cV_r).$$
Then $\cH_r$ is said to be the $r$th {\em left-definite space} associated with the pair $(\cH,{\bf A})$.
\end{defn}

It was proved in \cite[Theorem 3.1]{LW02} that $\cH_r=(\cV_r, \langle\fdot,\fdot\rangle)$ is also described as the left-definite space associated with the pair $(\cH, {\bf A}^r)$, and we call $\cH_r$ the {\em r}th {\em left-definite space associated with the pair} $(\cH,{\bf A})$.
Specifically, we have:
\begin{enumerate}
\item $\cH_r$ is a Hilbert space,
\item $\cD ({\bf A}^r)$ is a subspace of $\cV_r$,
\item $\cD ({\bf A}^r)$ is dense in $\cH_r$,
\item $\langle x,x\rangle_r\geq k^r\langle x,x\rangle$ ($x\in \cV_r$), and
\item $\langle x,y\rangle_r=\langle {\bf A}^rx,y\rangle$ ($x\in\cD ({\bf A}^r)$, $y\in \cV_r$).
\end{enumerate}

The left-definite domains are defined as the domains of compositions of the self-adjoint operator ${\bf A}$, but the operator acting on this domain is slightly more difficult to define. 

\begin{defn}
Let $\cH=(\cV,\langle\fdot,\fdot\rangle)$ be a Hilbert space. Suppose ${\bf A}:\cD ({\bf A})\subset \cH\to \cH$ is a self-adjoint operator that is bounded below by $k>0$. Let $r>0$. If there exists a self-adjoint operator ${\bf A}_r:\cH_r\to \cH_r$ that is a restriction of ${\bf A}$ from the domain $\cD({\bf A})$ to $\cD({\bf A}^r)$,
we call such an operator an $r$th {\em left-definite operator associated with $(\cH,{\bf A})$}.
\end{defn}

The connection between the $r$th left-definite operator and the $r$th composition of the self-adjoint operator ${\bf A}$ is now made explicit. 

\begin{cor}[{\cite[Corollary 3.3]{LW02}}] \label{t-comppower}
Suppose ${\bf A}$ is a self-adjoint operator in the Hilbert space $\cH$ that is bounded below by $k>0$. For each $r>0$, let $\cH_r=(\cV_r, \langle\fdot,\fdot\rangle_r)$ and ${\bf A}_r$ denote, respectively, the $r$th left-definite space and the $r$th left definite operator associated with $(\cH,{\bf A})$. Then
\begin{enumerate}
\item $\cD ({\bf A}^r)=\cV_{2r}$, in particular, $\cD ({\bf A}^{1/2})=\cV_1$ and $\cD ({\bf A})=\cV_2$;
\item $\cD ({\bf A}_r)=\cD ({\bf A}^{(r+2)/2})$, in particular, $\cD ({\bf A}_1)=\cD ({\bf A}^{3/2})$ and $\cD ({\bf A}_2)=\cD ({\bf A}^2)$.
\end{enumerate}
\end{cor}

The left-definite theory is particularly important for self-adjoint differential operators that are bounded below, as they are generally unbounded. The theory is trivial for bounded operators, as shown in {\cite[Theorem 3.4]{LW02}}.

Our applications of left-definite theory will be focused on differential operators which possess a complete orthogonal set of eigenfunctions in $\cH$. In {\cite[Theorem 3.6]{LW02}} it was proved that the point spectrum of ${\bf A}$ coincides with that of ${\bf A}_r$, and similarly for the continuous spectrum and for the resolvent set. 
Moreover, it turns out that a complete set of orthogonal eigenfunctions will persist throughout each space in the Hilbert scale.

\begin{theo}[{\cite[Theorem 3.7]{LW02}}] \label{t-leftdefortho}
If $\{\f_n\}_{n=0}^{\infty}$ is a complete orthogonal set of eigenfunctions of ${\bf A}$ in $\cH$, then for each $r>0$, $\{\f_n\}_{n=0}^{\infty}$ is a complete set of orthogonal eigenfunctions of the $r$th left-definite operator ${\bf A}_r$ in the $r$th left-definite space $\cH_r$.
\end{theo}

Another perspective on the last theorem is that it gives us a valuable indicator for when a space is a left-definite space for a specific operator.

On the side we note that left-definite theory can be extended to bounded below operators by applying shifts to create a semi-bounded operator. Uniqueness of left-definite domains are then given up to the chosen shift.

A description of these left-definite spaces in terms of standard boundary conditions on a Hilbert space has been noticeably missing, despite the broad framework and range of results described above. A previous paper of the authors \cite{FFL} attempted to remedy this deficiency, and gave both abstract and constructive approaches to the problem.

%%%%%%%%%%%%%%%%%%%%%%%%%%%%%
%%%%%%%%%%%%%%%%%%%%%%%%%%%%%
\subsection{Extension Theory}\label{ss-extensions}
%%%%%%%%%%%%%%%%%%%%%%%%%%%%%
%%%%%%%%%%%%%%%%%%%%%%%%%%%%%

There is a vast amount of literature concerning the extensions of symmetric operators. Here we present only what pertains to self-adjoint extensions and is related to our endeavors.

\begin{defn}[variation of {\cite[Section 14.2]{N}}]\label{d-defect}
For a a symmetric, closed operator ${\bf  A}$ on a Hilbert space $\cH$, define 
the {\bf positive defect space} and the {\bf negative defect space}, respectively, by
$$\cD_+:=\left\{f\in\cD({\bf  A}^*)~:~{\bf  A}^*f=if\right\}
\qquad\text{and}\qquad
\cD_-:=\left\{f\in\cD({\bf  A}^*)~:~{\bf  A}^*f=-if\right\}.$$
\end{defn}

We can assume without loss of generality that all considered operators are closed. This is because {\cite[Theorem XII.4.8]{DS}} says that the self-adjoint extensions of a symmetric operator coincide with those of the closure of the symmetric operator. 

We are particularly interested in the dimensions dim$(\cD_+)=m_+$ and dim$(\cD_-)=m_-$, which are called the {\bf positive} and {\bf negative deficiency indices of ${\bf  A}$}, respectively. These dimensions are usually conveyed as the pair $(m_+,m_-)$. 
The deficiency indices of ${\bf A}$ correspond to how ``far'' from self-adjoint ${\bf  A}$ is. A symmetric operator ${\bf  A}$ has self-adjoint extensions if and only if its deficiency indices are equal {\cite[Section 14.8.8]{N}}.

\begin{theo}[{\cite[Theorem 14.4.4]{N}}]\label{t-decomp}
If ${\bf  A}$ is a closed, symmetric operator, then the subspaces $\cD_{\bf  A}$, $\mathcal{D}_+$, and $\mathcal{D}_{-}$ are linearly independent and their direct sum coincides with $\cD_{{\bf  A}^*}$, i.e.,
$$\cD_{{\bf  A}^*}=\cD_{\bf  A}\dotplus\mathcal{D}_+ \dotplus\mathcal{D}_{-}.$$
(Here, subspaces $\cX_1, \cX_2, \hdots ,\cX_p$ are said to be {\bf linearly independent}, if $\sum_{i=1}^p x_i = 0$ for $x_i\in \cX_i$ implies that all $x_i=0$.)
\end{theo}

We now let $\ell[\fdot]$ be a Sturm--Liouville differential expression in order to introduce more specific definitions. It is important to reiterate that the analysis of self-adjoint extensions does not involve changing the differential expression associated with the operator at all, merely the domain of definition, by applying boundary conditions. 

\begin{defn}[{\cite[Section 17.2]{N}}]\label{d-max}
The {\bf maximal domain} of $\ell[\fdot]$ is given by 
\begin{align*}
\cD\ti{max}=\cD\ti{max}(\ell):=\bigg\{f:(a,b)\to\mathbb{C}~:~f,pf'\in\text{AC}\ti{loc}(a,b),
f,\ell[f]\in L^2[(a,b),w]\bigg\}.
\end{align*}
\end{defn}

The designation of ``maximal'' is appropriate in this case because $\cD\ti{max}(\ell)$ is the largest possible subspace for which $\ell$ maps back into $L^2[(a,b),w]$. For $f,g\in\cD\ti{max}(\ell)$ and $a<\al\le \beta<b$ the {\bf sesquilinear form} associated with $\ell$ by 
\begin{equation}\label{e-greens}
[f,g]\bigg|_{\al}^{\beta}:=\int_{\al}^{\beta}\left\{\ell[f(x)]\overline{g(x)}-\ell[\overline{g(x)}]f(x)\right\}w(x)dx.
\end{equation}

\begin{theo}[{\cite[Section 17.2]{N}}]\label{t-limits}
The limits $[f,g](b):=\lim_{x\to b^-}[f,g](x)$ and $[f,g](a):=\lim_{x\to a^+}[f,g](x)$ exist and are finite for $f,g\in\cD\ti{max}(\ell)$.
\end{theo}

The equation \eqref{e-greens} is {\bf Green's formula} for $\ell[\fdot]$, and in our case can be explicitly computed using integration by parts to be a modified Wronskian
\begin{align}\label{e-mwronskian}
[f,g]\bigg|_a^b:=p(x)[f'(x)g(x)-f(x)g'(x)]\bigg|_a^b.
\end{align}

\begin{defn}[{\cite[Section 17.2]{N}}]\label{d-min}
The {\bf minimal domain} of $\ell[\fdot]$ is given by
\begin{align*}
\cD^n\ti{min}=\cD\ti{min}(\ell)=\{f\in\cD\ti{max}(\ell)~:~[f,g]\big|_a^b=0~~\forall g\in\cD\ti{max}(\ell)\}.
\end{align*}
\end{defn}

The maximal and minimal operators associated with the expression $\ell[\fdot]$ are defined as ${\bf L}\ti{min}=\{\ell,\cD\ti{min}\}$ and ${\bf L}\ti{max}=\{\ell,\cD\ti{max}\}$ respectively. By {\cite[Section 17.2]{N}}, these operators are adjoints of one another, i.e.~$({\bf L}\ti{min})^*={\bf L}\ti{max}$ and $({\bf L}\ti{max})^*={\bf L}\ti{min}$.

In the context of differential operators, we work with the a special case of Theorem \ref{t-decomp}:

\begin{theo}[{\cite[Section 14.5]{N}}]\label{t-vN}
Let $\cD\ti{max}$ and $\cD\ti{min}$ be the maximal and minimal domains associated with the differential expression $\ell[\fdot]$, respectively. Then, 
\begin{equation}\label{e-vN}
\cD\ti{max}=\cD\ti{min}\dotplus\cD_+\dotplus\cD_-.
\end{equation}
\end{theo}

Equation \eqref{e-vN} is commonly known as {\bf von Neumann's formula}. Here $\dotplus$ denotes the direct sum, and $\cD_+,\cD_-$ are the defect spaces associated with the expression $\ell[\fdot]$. The decomposition can be made into an orthogonal direct sum by using the graph norm, see \cite{FFL}.

From {\cite[Section 14.8.8]{N}} we know that, if the operator ${\bf L}^n\ti{min}$ has any self-adjoint extensions, then the deficiency indices of ${\bf L}^n\ti{min}$ have the form $(m,m)$, where $0\leq m\leq 2n$ and $2n$ is the order of $\ell^n[\fdot]$.
Glazman \cite{G} proved that the number $m$ can take on any value between $0$ and $2n$. In regards to differential expressions, the order of the operator is greater than or equal to each of the two deficiency indices by necessity.  Hence, Sturm--Liouville expressions that generate self-adjoint operators have deficiency indices $(0,0)$, $(1,1)$ or $(2,2)$. This is related to the discussion of an expression being limit-point or limit-circle at endpoints, see \cite{AK, DS, N, W1, W2, W} for more details. 

The following theorem gives a concrete relation between the defect spaces and GKN boundary conditions, and is used to prove the main GKN theorems. To this end, let $\f_j$, for $j=1,\dots m$, denote an orthonormal basis of $\cD_+$. The functions $\overline{\f_j}$ are thus an orthonormal basis of $\cD_-.$

\begin{theo}[{\cite[Theorem 18.1.2]{N}}]\label{t-gknmatrix}
Every self-adjoint extension ${\bf L}^n=\{\ell^n,\cD_{{\bf L}}^n\}$ of the minimal operator ${\bf L}^n\ti{min}=\{\ell^n,\cD^n\ti{min}\}$ with deficiency indices $(m,m)$ can be characterized by means of a unitary $m\times m$ matrix $u=[u_{jk}]$ in the following way:

Its domain of definition $\cD_{{\bf L}}^n$ is the set of all functions $z(x)$ of the form
$$ z(x)=y(x)+\psi(x), $$
where $y(x)\in\cD^n\ti{min}$ and $\psi(x)$ is a linear combination of the functions
$$ \psi_j(x)=\f_j(x)+\sum_{k=1}^m u_{kj}\overline{\f_k(x)}, \quad j=1,\dots,m. $$
Conversely, every unitary $m\times m$ matrix $u=[u_{jk}]$ determines (in the way described above) a certain self-adjoint extension ${\bf L}^n$ of the operator ${\bf L}^n\ti{min}$. The correspondence thus established between ${\bf L}^n$ and $u$ is one-to-one.
\end{theo}

In order to formulate the GKN Theorems, we recall an extension of linear independence to one that mods out by a subspace. This subspace will be the minimal domain in applications.

\begin{defn}[{\cite[Section 14.6]{N}}]\label{d-linind}
Let $\cX_1$ and $\cX_2$ be subspaces of a vector space $\cX$ such that $\cX_1\le \cX_2$. Let $\{x_1,x_2,\dots,x_r\}\subseteq \cX_2$. We say that $\{x_1,x_2,\dots,x_r\}$ is {\bf linearly independent modulo $\cX_1$} if
$$\sum_{i=1}^r\al_ix_i\in \cX_1 \text{ implies } \al_i=0\text{ for all }i=1,2,\dots, r.$$
\end{defn}

The following two theorems form the core of GKN theory.

\begin{theo}[GKN1,~{\cite[Theorem 18.1.4]{N}}]\label{t-gkn1}
Let ${\bf L}^n=\{\ell^n,\cD_{{\bf L}}^n\}$ be a self-adjoint extension of the minimal operator ${\bf L}^n\ti{min}=\{\ell^n,\cD^n\ti{min}\}$ with deficiency indices $(m,m)$ and associated sesquilinear form $[\cdot,\cdot]_n$. Then the domain $\cD_{{\bf L}}^n$ consists of the set of all functions $f\in\cD^n\ti{max}$, which satisfy the conditions
\begin{equation}\label{e-gkn1a}
[f,w_k]_n\bigg|_a^b=0, \text{ }k=1,2,\dots,m ,
\end{equation}
where $w_1,\dots,w_m\in \cD^n\ti{max}$ are linearly independent modulo $\cD^n\ti{min}$ for which the relations
\begin{equation}\label{e-gkn1b}
[w_j,w_k]_n\bigg|_a^b=0, \text{ }j,k=1,2,\dots,m
\end{equation}
hold.
\end{theo}

The requirements in equation \eqref{e-gkn1b} are commonly referred to as {\bf Glazman symmetry conditions}. The converse of the GKN1 Theorem is also true.

\begin{theo}[GKN2,~{\cite[Theorem 18.1.4]{N}}]\label{t-gkn2}
Assume we are given arbitrary functions $w_1,w_2,\dots,w_m\in\cD^n\ti{max}$ which are linearly independent modulo $\cD^n\ti{min}$ and which satisfy the relations \eqref{e-gkn1b}. Then the set of all functions $f\in\cD^n\ti{max}$ which satisfy the conditions \eqref{e-gkn1a} is domain of a self-adjoint extension of ${\bf L}^n\ti{min}$.
\end{theo}

These two theorems completely answer the question of how boundary conditions can be used to create self-adjoint extensions. Applications of this theory hinge on determining the proper $w_k$'s that will define the domain of the desired self-adjoint extension.

%%%%%%%%%%%%%%%%%%%%%%%%%%%%%
%%%%%%%%%%%%%%%%%%%%%%%%%%%%%
\section{Domains of Sturm--Liouville Operators}\label{s-mindomains}
%%%%%%%%%%%%%%%%%%%%%%%%%%%%%
%%%%%%%%%%%%%%%%%%%%%%%%%%%%%

We begin by discussing domains associated with the Legendre and Jacobi differential operators, with a slightly different perspective on each. The simple results in this section mainly serve as an introduction to the methods of Section \ref{s-powers}, where analysis is more complicated. Some results may also be derived from existing literature, such as \cite[Section 3]{HK}, but methods of proof are different.

%%%%%%%%%%%%%%%%%%%%%%%%%%%%%
%%%%%%%%%%%%%%%%%%%%%%%%%%%%%
\subsection{The Legendre Operator}\label{ss-legendreexample}
%%%%%%%%%%%%%%%%%%%%%%%%%%%%%
%%%%%%%%%%%%%%%%%%%%%%%%%%%%%

As a first example, we consider the classical Legendre differential expression given by
\begin{align}\label{e-legendre}
\ell[f](x)=-((1-x^2)f'(x))'
\end{align}
on the maximal domain
\begin{align}\label{e-legendremax}
\cD\ti{max} &= \{f:(-1,1)\to \CC
\text{ }:\text{ } f, f'\in\text{AC}\ti{loc}(-1,1); f,\ell[f]\in L^2(-1,1)
\}.
\end{align}
This maximal domain defines the associated minimal domain given in Definition \ref{d-min}, and the defect indices are $(2,2)$, with both endpoints in the limit-circle case. The symmetric expression given in equation \eqref{e-legendre} possesses the Legendre polynomials $P_m(x)$, $m\in\NN_0$, as a complete in $L^2(-1,1)$ set of functions $f(x)=P_m(x)$ that satisfy the eigenvalue equation 
\begin{align*}
\ell[f](x)=m(m+1)f(x)
\end{align*}
for each $k$. We also recall that Legendre polynomials \cite[Section 14.3]{DLMF} can be written as:
\begin{align}\label{e-legendrehyper}
P_m(x)=\,_2F_1(-m,m+1;1;(1-x)/2)=\sum_{j=0}^m\left(-\dfrac{1}{2}\right)^j{m \choose j}{m+j \choose j}(1-x)^j,
\end{align}
where $\,_2F_1(a,b;c;x)$ is the Gauss hypergeometric series. Of course, the series can be written with powers of $(1+x)$ using a transformation such as \cite[Equation 15.8.7]{DLMF}. The Legendre polynomial $P_m(x)$ is in $\cD\ti{max}$ for each $m\in\NN_0$.

The associated sesquilinear form is defined, for $f,g\in\cD\ti{max}$, via equation \eqref{e-greens}. Integration by parts easily yields the explicit expression
\begin{align*}
    [f,g](\pm 1)=\lim_{x\to 1^{\mp}}(1-x^2)[f'(x)g(x)-f(x)g'(x)].
\end{align*}
Theorem \ref{t-limits} says that the sesquilinear form is both well-defined and finite for all $f,g\in\cD\ti{max}$. We will use specific choices of the function $g(x)$ to say that this imposes a certain degree of regularity on $f(x)$ and $f'(x)$.

\begin{prop}\label{t-legendremax}
Let $f\in\cD\ti{max}$. Then,
\begin{align*}
    \lim_{x\to \pm 1^{\mp}}(1-x^2)f(x)=0, \text{ and } \lim_{x\to\pm 1^{\mp}}(1-x^2)f'(x) \text{ is finite.}
\end{align*}
\end{prop}

\begin{proof}
We prove the proposition at the endpoint $1$, and the result at $-1$ will similarly follow. 

Let $f\in\cD\ti{max}$. The function $1\in\cD\ti{max}$ trivially, so
\begin{align*}
[f,1](1)=\lim_{x\to 1^-}(1-x^2)f'(x) \text{ is finite}.
\end{align*}
Hence,
\begin{align}\label{e-imp1higher}
\lim_{x\to 1^-}(1-x^2)f'(x)(1-x)=\lim_{x\to 1^-}(1-x^2)f'(x)\cdot\lim_{x\to 1^-}(1-x)=0.
\end{align}
The function $(1-x)$ is also clearly in $\cD\ti{max}$, so by equation \eqref{e-imp1higher}
\begin{align*}
[f,1-x](1)=\lim_{x\to 1^-}(1-x^2)[-f(x)-f'(x)(1-x)]\approx\lim_{x\to 1^-}(1-x^2)f(x) \text{ is finite.}
\end{align*}
Without loss of generality, let $c$ be a nonzero positive constant such that
\begin{align*}
    c:=\lim_{x\to 1^-}(1-x^2)f(x).
\end{align*}
If $c$ is negative, simply take the absolute value of both sides. Define $r:=c/2$ so that 
\begin{align*}
r<\lim_{x\to 1^-}(1-x^2)f(x).
\end{align*}
Dividing both sides by $(1-x^2)$ yields 
\begin{align}\label{e-dividinglegendre}
\lim_{x\to 1^-}\dfrac{r}{(1-x^2)}<\lim_{x\to 1^-}f(x).
\end{align}
The definition of the maximal domain says that $f\in L^2(-1,1)$, so $f\in L^1(-1,1)$ by the embedding of $L^p$ spaces. Integrating both sides of equation \eqref{e-dividinglegendre} thus yields a contradiction to the fact that $f\in L^1(-1,1)$ by the comparison test. As $c$ was an arbitrary constant, we conclude that $c=0$.
\end{proof}

We now define a class of $C^{\infty}[-1,1]$ functions, for $j\in\NN_0$, by their behavior near the endpoints. We take representative elements of this class, for $j\in\NN_0$, to be denoted by
\begin{equation}\label{e-firstkind}
\begin{aligned}
\f_j^+&:=\left.
\begin{cases}
(1-x)^j, & \text{ for }x \text{ near }1 \\
0, & \text{ for }x \text{ near }-1
\end{cases}
\right\}, \text{ and} \\
\f_j^-&:=\left.
\begin{cases}
0, & \text{ for }x \text{ near }1 \\
(1+x)^j, & \text{ for }x \text{ near }-1
\end{cases}
\right\}.
\end{aligned}
\end{equation}
Note that the functions $\f_0^+$ and $\f_0^-$ simply behave like the function $1$ near the endpoints $1$ and $-1$ respectively. It is also clear that the functions are in the maximal domain $\cD\ti{max}$. This class of functions will be shown to play an important role within the maximal domains associated with powers of the Legendre and Jacobi differential equations.

\begin{cor}\label{t-legendre+}
For $j\geq 1$, the functions $\f_j^+$, $\f_j^-$ belong to the minimal domain associated with the Legendre differential expression \eqref{e-legendre}.
\end{cor}

\begin{proof}
The definition of the minimal domain (Definition \ref{d-min}) means the theorem is equivalent to showing that 
\begin{align*}
\left[f,\f_j^+\right]\bigg|_{-1}^1=0, \quad \forall f\in\cD\ti{max}.
\end{align*}
Hence, for all $f\in\cD\ti{max}$ and $j\geq 1$,
\begin{align*}
\left[ f,(1-x)^j\right](1)=\lim_{x\to 1^-}(1-x^2)\left\{f(x)\left[-j(1-x)^{j-1}\right]-f'(x)(1-x)^j\right\}.
\end{align*}
The two conclusions of Proposition \ref{t-legendremax} together imply that the limit is 0, as desired.
The result follows for the functions $\f_j^-$ analogously.
\end{proof}

The characterizations of the maximal domain in Proposition \ref{t-legendremax} and the minimal domain in Corollary \ref{t-legendre+} are related. If functions in the maximal domain possess only a certain amount of blow-up, then the minimal domain will include functions with regularity sufficient to eliminate such singularities. 

In particular, the maximal domain is rarely rewritten in the literature (see e.g.~\cite[Theorem 10.1]{ELM}, and \cite{HK,P2}), but to the best knowledge of the authors, the properties in Proposition \ref{t-legendremax} are new. Indeed, investigations of the maximal domain are usually only done for the purpose of finding specific self-adjoint descriptions. This goal becomes much easier when additional properties of the maximal and minimal domains are known.

In order to show this, consider the self-adjoint operator ${\bf T}$ generated via expression \eqref{e-legendre} that contains the Legendre polynomials (see e.g.~\cite{LW02, LW15, LWOG}). The domain of this operator can be written as 
\begin{align*}
\cD_{\bf T}=\left\{f\in\cD\ti{max}~:~[f,1](\pm 1)=0\right\}.
\end{align*}
 
The deficiency indices of $\ell$ are well-known to be $(2,2)$ so two boundary conditions are imposed, one at each endpoint \cite{BEZ, EKLWY}. There are many characterizations of the operator $T$ in  literature, but the following one is particularly insightful for us.

\begin{lem}[{\cite[Theorem 7.1]{ELM}}]\label{t-AC} 
Let $f\in\cD\ti{max}$. Then $f\in\cD\ti{\bf T}$ if and only if $f\in AC[-1,1]$. In particular, $f\in\cD\ti{\bf T}$ implies
\begin{align*}
    \lim_{x\to\pm 1^{\mp}}f(x) \text{ is finite}.
\end{align*}
\end{lem}

This condition is comparable to that of Proposition \ref{t-legendremax}. Indeed, the juxtaposition of these two results is one of the central themes of our investigation of domains for higher order operators. The maximal domain will ensure that some limits of derivatives are finite when a certain degree of regularity is present. On the other hand, inclusion in the left-definite domain ensures that limits of the same derivatives with \emph{half} as much regularity are finite. The distinction between which limits are 0 and which are finite will be the center of discussion in Section \ref{s-equivalence}.

For contrast, we define the domain
\begin{align}
\cD\ti{{\bf L}(m)}=\left\{ f\in\cD\ti{max}~:~[f,P_m](\pm 1)=0\right\},
\end{align}
where $P_m$ is the $m$th Legendre polynomial, $m\in\NN_0$. The following surprising theorem can be formulated.

\begin{theo}\label{t-anylegendre}
The self-adjoint operators ${\bf T}=\{\ell,\cD\ti{\bf T}\}$ and ${{\bf L}(m)}=\{\ell,\cD\ti{{\bf L}(m)}\}$ are equal.
\end{theo}

\begin{proof}
The definition is equation \eqref{e-legendrehyper} says $P_m(x)$, $m\in\NN_0$, can be decomposed as
\begin{align*}
P_m(x)=\sum_{j=0}^m a_j(1-x)^j,
\end{align*}
where $a_j\in\RR$ are constants. The case $m=0$ follows trivially. Fix $m\in\NN$ and let $f\in\cD\ti{\bf T}$. Then, by linearity of the sesquilinear form,
\begin{align*}
[f,P_m](1)=\sum_{j=0}^m a_j[f,(1-x)^j](1)=a_0[f,1](1)=0.
\end{align*}
All but one term vanished due to the fact that $\f_j^+$, $j\geq 1$, is in the minimal domain by Corollary \ref{t-legendre+}. The final equality holds because of the definition of $\cD\ti{\bf T}$.
Similarly, if we write equation \eqref{e-legendrehyper} in terms of $(1+x)$ with new constants $b_j$ then
\begin{align*}
[f,P_m](-1)=\sum_{j=0}^m b_j[f,(1+x)^j](-1)=b_0[f,1](-1)=0.
\end{align*}
Hence, $f\in\cD_{{\bf L}(n)}$. The reverse inclusion follows analogously, and the theorem is proven.
\end{proof}

It should be noted that Theorem \ref{t-anylegendre} is not new to the literature, although it's proof uses the new properties of the maximal domain and is more efficient. The result appears in a previous paper of the authors \cite{FFL} that focused on which GKN boundary conditions can be utilized to describe the left-definite domain. Therein, a constructive approach is applied to determine that even for the square and cube of the Legendre operator, an appropriate amount of any Legendre polynomials can be chosen as GKN boundary conditions and still yield the left-definite domain.

The Legendre differential expression is a special case of the Jacobi differential expression, with $\al=\beta=0$. It is therefore natural to consider whether the previous results hold for more general parameters.

%%%%%%%%%%%%%%%%%%%%%%%%%%%%%
%%%%%%%%%%%%%%%%%%%%%%%%%%%%%
\subsection{The Jacobi Operator}\label{ss-jacobiexample}
%%%%%%%%%%%%%%%%%%%%%%%%%%%%%
%%%%%%%%%%%%%%%%%%%%%%%%%%%%%

 Let $0\leq\al,\beta<1$, and consider the classical Jacobi differential expression given by
\begin{align}\label{e-jacobi}
\ell_{\al,\beta}[f](x)=-\dfrac{1}{(1-x)^{\al}(1+x)^{\beta}}[(1-x)^{\al+1}(1+x)^{\beta+1}f'(x)]'
\end{align}
on the maximal domain
\begin{align*}
\cD^{(\al,\beta)}\ti{max}=\{ f\in L^2_{\al,\beta}(-1,1) ~|~ f,f'\in AC\ti{loc};\ell_{\al,\beta}[f]\in L^2_{\al,\beta}(-1,1)\},
\end{align*}
where the Hilbert space $L^2_{\al,\beta}(-1,1):=L^2\left[(-1,1),(1-x)^{\al}(1+x)^{\beta}\right]$. This maximal domain defines the associated minimal domain given in Definition \ref{d-min}, and the defect indices are $(2,2)$. The specified values of $\al,\beta$ will ensure that the differential expression is in the limit circle case at both endpoints, and so are assumed throughout. If either parameter is equal to or larger than 1, then all of our conclusions still hold, but some boundary conditions will be satisfied trivially. If either are less than 0, the corresponding endpoint is regular and although it still requires a boundary condition, these are much simpler and don't need to be in the GKN format. See \cite{BEZ, EKLWY, Z} for more details.

The the Jacobi polynomials $P_m^{(\al,\beta)}(x)$, $m\in\NN_0$, are complete in $L^2_{\al,\beta}(-1,1)$ set for which $f(x)=P_m^{(\al,\beta)}(x)$ solves the eigenvalue equation of the symmetric expression given in equation \eqref{e-jacobi}, that is:
\begin{align*}
\ell_{\al,\beta}[f](x)=m(m+\al+\beta+1)f(x)
\end{align*}
for each $m$. We also recall that Jacobi polynomials \cite[Section 18.5]{DLMF} can be written as:
\begin{align}\label{e-jacobihyper}
P_m^{(\al,\beta)}(x)&=\dfrac{(\al+1)_m}{m!}\,_2F_1(-m,m+\al+\beta+1;\al+1;(1-x)/2) \nonumber \\
&=\sum_{j=0}^m\left(-\dfrac{1}{2}\right)^j\dfrac{(m+\al+\beta+1)_j(\al+j+1)_{m-j}}{m!(m-j)!}(1-x)^j,
\end{align}
where $\,_2F_1(a,b;c;x)$ is the Gauss hypergeometric series and $(\cdot)_j$ denotes the rising Pochhammer symbol. Of course, the series can be written with powers of $(1+x)$ using a transformation such as \cite[Equation 15.8.7]{DLMF}. The Jacobi polynomial $P_m^{(\al,\beta)}(x)$ is in $\cD^{(\al,\beta)}\ti{max}$ for each $m\in\NN_0$.

The associated sesquilinear form is defined, for $f,g\in\cD\ti{max}^{(\al,\beta)}$, via equation \eqref{e-greens}. Integration by parts easily yields the explicit expression
\begin{align*}
[f,g]_1(\pm 1):=\lim_{x\to\pm 1^{\mp}}(1-x)^{\al+1}(1+x)^{\beta+1}[f'(x)g(x)-f(x)g'(x)].
\end{align*}
Note that the dependence of the sesquilinear form on the parameters $\al$ and $\beta$ is suppressed in the definition for the sake of notation.

Theorem \ref{t-limits} says that the sesquilinear form is both well-defined and finite for all $f,g\in\cD\ti{max}^{(\al,\beta)}$. We will use specific choices of the function $g(x)$ to say that this imposes a certain degree of regularity on $f(x)$ and $f'(x)$.
 
\begin{prop}\label{t-jacobimax}
Let $f\in\cD\ti{max}^{(\al,\beta)}$. Then,
\begin{align*}
    \lim_{x\to \pm 1^{\mp}}(1-x)^{\al+1}(1+x)^{\beta+1}f(x)=0, \text{ and } \lim_{x\to\pm 1^{\mp}}(1-x)^{\al+1}(1+x)^{\beta+1}f'(x) \text{ is finite}.
\end{align*}
\end{prop}

The proof is analogous to that of Proposition \ref{t-legendremax}, but is repeated here because it will be used as a base case for induction in the proof of Theorem \ref{t-maxdomaindecomp}.

\begin{proof}
Consider the endpoint $1$.
Let $f\in\cD\ti{max}^{(\al,\beta)}$. The function $1\in\cD\ti{max}^{(\al,\beta)}$ trivially, so
\begin{align*}
[f,1]_1(1)=\lim_{x\to 1^-}(1-x)^{\al+1}f'(x) \text{ is finite}.
\end{align*}
Hence, 
\begin{align}\label{e-jimp1}
\lim_{x\to 1^-}(1-x)^{\al+1}f'(x)(1-x)=\lim_{x\to 1^-}(1-x)^{\al+1}f'(x)\cdot\lim_{x\to 1^-}(1-x)=0.
\end{align}
The function $(1-x)$ is also clearly in $\cD\ti{max}^{(\al,\beta)}$, so by equation \eqref{e-jimp1}
\begin{align*}
[f,1-x]_1(1)=\lim_{x\to 1^-}(1-x)^{\al+1}[-f(x)-f'(x)(1-x)]\approx\lim_{x\to 1^-}(1-x)^{\al+1}f(x) \text{ is finite}.
\end{align*}
Without loss of generality, let $c$ be a nonzero positive constant such that
\begin{align*}
    c:=\lim_{x\to 1^-}(1-x)^{\al+1}f(x).
\end{align*}
If $c$ is negative, simply take the absolute value of both sides. Define $r:=c/2$ so that 
\begin{align*}
r<\lim_{x\to 1^-}(1-x)^{\al+1}f(x).
\end{align*}
Dividing both sides by $(1-x)$ yields 
\begin{align}\label{e-dividingjacobi}
\lim_{x\to 1^-}\dfrac{r}{(1-x)}<\lim_{x\to 1^-}(1-x)^{\al}f(x).
\end{align}

The definition of the maximal domain says that $f\in L^2_{\al,\beta}(-1,1)$, so $f\in L^1_{\al,\beta}(-1,1)$ by the embedding of $L^p$ spaces. Integrating both sides of equation \eqref{e-dividingjacobi} thus yields a contradiction to the fact that $f\in L^1_{\al,\beta}(-1,1)$ by the comparison test. As $c$ was an arbitrary constant, we conclude that $c=0$.
The results at the endpoint $-1$ similarly follow.
\end{proof}

\begin{cor}\label{t-jacobi1}
For $j\geq 1$, the functions $\f_j^+$ and $\f_j^-$ belong to the minimal domain associated with the Jacobi differential expression \eqref{e-jacobi}. 
\end{cor}
\begin{proof}
The proof is analogous to that of Corollary \ref{t-legendre+}, but relying on Proposition \ref{t-jacobimax} instead of Proposition \ref{t-legendremax}.
\end{proof}

It is somewhat surprising that the parameters $\al$ and $\beta$ don't distinguish Corollary \ref{t-jacobimax} from Corollary \ref{t-legendremax} in the Legendre case. The range of the parameters seems to be the culprit, as $0\leq\al,\beta<1$. Hence, when the maximal domain is considered in the basis of Jacobi polynomials at each endpoint, the powers of $(1-x)$ and $(1+x)$ included in the minimal domain should be higher than $\al$ and $\beta$, respectively. In the Legendre case we have $\al=\beta=0$, so only the function $1$ is outside the minimal domain. This idea of a rough polynomial \emph{boundary} that separates the minimal from the maximal domain also reflects the duality described between the two operators mentioned after Corollary \ref{t-legendre+}. A proof of the sharpness of this boundary is still outstanding.

Now, consider the self-adjoint operator ${\bf J}^{(\al,\beta)}$ generated via expression \eqref{e-jacobi} that contains all of the Jacobi polynomials, $P_m^{(\al,\beta)}$. The domain of this operator can be written \cite[Section 3]{EKLWY} as 
\begin{align*}
  \cD_{{\bf J}^{(\al,\beta)}}=\left\{f\in\cD\ti{max}^{(\al,\beta)}~:~[f,1]_1(\pm 1)=0\right\}.
\end{align*}
Following the convention of \cite{FFL}, the deficiency indices of $\ell_{\al,\beta}$ are $(1,1)$, so one boundary condition is imposed. The domain has a property that represents a generalization of Lemma \ref{t-AC}.

\begin{lem}\cite[Section 3]{EKLWY}\label{t-stronglimitpoint}
The domain $\cD_{{\bf J}^{(\al,\beta)}}$ possesses the strong limit-point condition. That is, for all $f\in\cD_{{\bf J}^{(\al,\beta)}}$, we have
\begin{align*}
\lim_{x\to\pm 1^{\mp}}(1-x)^{\al+1}(1+x)^{\beta+1}f(x)\overline{g}'(x)=0,
\end{align*}
for all $g\in\cD_{{\bf J}^{(\al,\beta)}}$.
\end{lem}

It should be noted that Lemma \ref{t-stronglimitpoint} also applies to the Legendre differential expression, but the stronger Lemma \ref{t-AC} was used instead. The comparison of Lemma \ref{t-stronglimitpoint} and Proposition \ref{t-jacobimax} thus offers another reference point to the duality between the maximal and minimal domain. 

It is possible to recover the first conclusion of Proposition \ref{t-jacobimax} from Lemma \ref{t-stronglimitpoint} by setting $g(x)=x\in\cD\ti{max}^{(\al,\beta)}$. This yields that, for $f\in\cD_{{\bf J}^{(\al,\beta)}}$,
\begin{align*}
\lim_{x\to\pm 1^{\mp}}(1-x)^{\al+1}(1+x)^{\beta+1}f(x)=0.
\end{align*}

However, the description of the maximal domain in Proposition \ref{t-jacobimax} should not be considered a generalization of the strong limit-point condition (for more see e.g.~\cite{E-SLP}). The strong-limit point condition is achieved by imposing a boundary condition on the maximal domain, and hence it only holds for a specific self-adjoint extension. The maximal domains of powers of the Jacobi expression and their associated left-definite domains in Section \ref{s-equivalence} have a similar relationship.

For contrast, we define the domain
\begin{align}
\cD_{{\bf J}^{(\al,\beta)}(m)}=\left\{ f\in\cD\ti{max}^{(\al,\beta)}~:~\left[f,P_m^{(\al,\beta)}\right]_1(\pm 1)=0\right\},
\end{align}
where $P_m^{(\al,\beta)}$ is the $m$th Jacobi polynomial, $m\in\NN_0$.

\begin{theo}\label{t-firstjacobiany}
The self-adjoint operators defined as ${{\bf J}^{(\al,\beta)}}=\left\{\ell,\cD_{{\bf J}^{(\al,\beta)}} \right\}$ and ${{\bf J}^{(\al,\beta)}(m)}=\left\{ \ell,\cD_{{\bf J}^{(\al,\beta)}(m)} \right\}$ are equal.
\end{theo}

\begin{proof}
The definition in equation \eqref{e-jacobihyper} says $P_m^{(\al,\beta)}(x)$, $m\in\NN_0$, can be decomposed as
\begin{align*}
P_m^{(\al,\beta)}(x)=\sum_{j=0}^m a_j(1-x)^j,
\end{align*}
where $a_j\in\RR$ are constants. The case $m=0$ is trivial. Fix $m\in\NN$ and let $f\in\cD_{{\bf J}^{(\al,\beta)}}$. Then, by linearity of the sesquilinear form
\begin{align*}
\left[f,P_m^{(\al,\beta)}\right]_1(1)=\sum_{j=0}^m a_j\left[f,(1-x)^j\right]_1(1)
=a_0[f,1]_1(1)=0.
\end{align*}
All but one term vanished due to the fact that $\f_j^+$, $j\geq 1$, is in the minimal domain by Corollary \ref{t-jacobi1}. The final equality holds because of the definition of $f\in\cD_{{\bf J}^{(\al,\beta)}}$.

Similarly, if we write equation \eqref{e-jacobihyper} in terms of $(1+x)$ with new constants $b_j$ then
\begin{align*}
\left[f,P_m^{(\al,\beta)}\right]_1(-1)&=\sum_{j=0}^m b_0[f,1]_1(-1)=0.
\end{align*}
Hence, $f\in\cD_{{\bf J}^{(\al,\beta)}(m)}$. The reverse inclusion follows analogously, and the theorem is proven.
\end{proof}

We now attempt to generalize these domain decompositions to those associated with powers of the Jacobi differential expression.

%%%%%%%%%%%%%%%%%%%%%%%%%%%%%
%%%%%%%%%%%%%%%%%%%%%%%%%%%%%
\section{Domains of Powers of the Jacobi Operator}\label{s-powers}
%%%%%%%%%%%%%%%%%%%%%%%%%%%%%
%%%%%%%%%%%%%%%%%%%%%%%%%%%%%

Let $0\leq\al,\beta<1$ and $n\in\NN$. It is known \cite{EKLWY} that the $n$th composition of the Jacobi differential expression \eqref{e-jacobi} can be expressed in Lagrangian symmetric form as
\begin{align}\label{e-njacobi}
\ell_{{\bf J}}^n[f](x)=-\dfrac{1}{(1-x)^{\al}(1+x)^{\beta}}\sum_{k=1}^n (-1)^k[C(n,k,\al,\beta)(1-x)^{\al+k}(1+x)^{\beta+k}f^{(k)}(x)]^{(k)},
\end{align}
on the maximal domain
\begin{align*}
\cD^{{\bf J}, n}\ti{max}=\{ f\in L^2_{\al,\beta}(-1,1) ~|~ f,f',\dots,f^{(2n-1)}\in AC\ti{loc}(-1,1);\ell_{\al,\beta}^n[f]\in L^2_{\al,\beta}(-1,1)\},
\end{align*}
where the Hilbert space $L^2_{\al,\beta}(-1,1)=L^2\left[(-1,1),(1-x)^{\al}(1+x)^{\beta}\right]$. This maximal domain defines the associated minimal domain given in Definition \ref{d-min} and the deficiency indices of $\ell_{{\bf J}}^n$ are $(2n,2n)$. To make the notation more accessible, we are suppressing the dependence on $\al$ and $\beta$ in the definition of $\cD^{{\bf J}, n}\ti{max}$, $\cD^{{\bf J}, n}\ti{min}$ and the defect spaces $\cD_+^{{\bf J}, n}$, $\cD_-^{{\bf J}, n}$, see equation \eqref{e-vN}. Explicit values for the constants $C(n,k,\al,\beta)$ can also be found in \cite{EKLWY}.

The associated sesquilinear form is defined, for $f,g\in\cD^{{\bf J},n}\ti{max}$, via equation \eqref{e-greens}. It can be written explicitly \cite[Section 6]{FFL} as
\begin{align}\label{e-njacobisesqui}
[f,g]_n(x):=\sum_{k=1}^n\sum_{j=1}^k(-1)^{k+j}\bigg\{\big[a_k(x)\overline{g}^{(k)}(x)\big]^{(k-j)}&f^{(j-1)}(x)-\\
&\big[a_k(x)f^{(k)}(x)\big]^{(k-j)}\overline{g}^{(j-1)}(x)\bigg\}\nonumber,
\end{align}
where $a_k(x)=(1-x)^{\al+k}(1+x)^{\beta+k}$. 

The functions $(1-x)^j$ and $(1+x)^j$, for finite $j\in\NN_0$, are in $\cD\ti{max}^{{\bf J},n}$ because they are polynomials which are bounded at the endpoints. Equation \eqref{e-njacobi} allows us to easily calculate
\begin{align}\label{e-basisfunctionsinmax}
    \ell^n_{{\bf J}}[(1-x)^j]=\sum_{k=1}^n\sum_{i=0}^k c(n,k,\al,\beta,i,j)(1-x)^{j-k+i}(1+x)^{k-i}
\end{align}
 for some finite constants $c(n,k,\al,\beta,i,j)$, which vanish when $i-k>j$. The expression for $\ell^n_{{\bf J}}[(1+x)^j]$ can be determined analogously and is also a polynomial bounded on $[-1,1]$. 
 Theorem \ref{t-limits} says that, for $f,g\in\cD^{{\bf J},n}\ti{max}$, equation \eqref{e-njacobisesqui} is finite. The following theorem uses the functions $(1-x)^j$ and $(1+x)^j$ in the sesquilinear form to determine properties of functions in the maximal domain.

\begin{theo}\label{t-maxdomaindecomp}
If $f\in\cD\ti{max}^{{\bf J},n}$, then for $j=0,\dots,n$ 
\begin{align}\label{e-decompresult1}
    \lim_{x\to \pm 1^{\mp}}(1-x)^{\al+j}(1+x)^{\beta+j}f^{(j)}(x) \text{ is finite}.
\end{align}
Furthermore, for $k=0,\dots,n$ and each $j\in\NN$ such that $j<k$,
\begin{align}\label{e-decompresult2}
\lim_{x\to \pm 1^{\mp}}\left[(1-x)^{\al+k}(1+x)^{\beta+k}f^{(k)}(x)\right]^{(k-j)} \text{ is finite}.
\end{align}
\end{theo}

\begin{proof}
The case $n=1$ was proven in Proposition \ref{t-jacobimax}. Without loss of generality, fix $n\in\NN$ such that $n\geq 2$ and let $f\in\cD\ti{max}^{{\bf J},n}$. In order to avoid tedious repetition, we will only include the argument for the endpoint $1$. Hence, functions of the form $(1+x)^{\beta+k}$ are suppressed here and throughout the calculations, as they contribute only a finite constant at that endpoint. The results at $-1$ follow by an analogous argument.

Notice that, by definition, the maximal domains for different compositions are nested and form the chain
\begin{align}\label{e-maxnesting}
\cdots\subset\cD\ti{max}^{{\bf J},n}\subset\cdots\subset\cD\ti{max}^{{\bf J},2}\subset\cD\ti{max}^{{\bf J}}.
\end{align}

Theorem \ref{t-limits} says that $[f,g]_n(-1)$ and $[f,g]_n(1)$ both exist and are finite for all $g\in\cD\ti{max}^{{\bf J},n}$. The chain in equation \eqref{e-maxnesting} then implies that $[f,g]_m(-1)$ and $[f,g]_m(1)$ exist and are finite for all $m\in\NN_0$ such that $m\leq n$.

The main argument now includes a nested sequence of induction proofs. We claim that, for $s\in\NN$ and $r\in\NN_0$ such that $s \leq r \leq n-1$,
\begin{align}\label{e-inductionclaimgeneral}
    \lim_{x\to 1^-}\left[(1-x)^{\al+r}f^{(r)}(x)\right]^{(r-s)} \text{ finite } \implies \lim_{x\to 1^-}\left[(1-x)^{\al+r}f^{(r)}(x)\right]^{(r-(s+1))} \text{ finite}.
\end{align}
On the side, we mention that when $r=s$ the statement has a special case that is also part of our goal for the Theorem:
\begin{align}\label{e-specialcase}
 \lim_{x\to 1^-}(1-x)^{\al+r}f^{(r)}(x) \text{ is finite}.
\end{align}
Proceed by induction on $s$. Each value of $s$ will also require a proof by induction, so begin with the base case for $s=1$. 

We claim that for $r=1,\dots,n$
\begin{align}\label{e-inductionclaim1}
    \lim_{x\to 1^-}\left[(1-x)^{\al+r}f^{(r)}(x)\right]^{(r-1)} \text{ is finite}.
\end{align}
The base case where $r=1$ was shown in Proposition \ref{t-jacobimax}, and applies via equation \eqref{e-maxnesting}. For the sake of clarity, we provide the argument when $r=2$ as well. We wish to examine $[f,1]_m(1)$ for higher values of $m$. Note that the limits in Proposition \ref{t-jacobimax} are still finite when multiplied by powers of $(1-x)$. Thus,
\begin{align}
[f,1]_2(1)&=\lim_{x\to 1^-}\sum_{k=1}^2\sum_{j=1}^k(-1)^{k+j}\left[(1-x)^{\al+k}f^{(k)}(x)\right]^{(k-j)}(1)^{(j-1)}\\
&=\lim_{x\to 1^-}(1-x)^{\al+1}f'(x)-\left[(1-x)^{\al+2}f''(x)\right]'\nonumber \text{ is finite},
\end{align}
which implies that 
\begin{align}\label{e-1basecase}
    \lim_{x\to 1^-}\left[(1-x)^{\al+2}f''(x)\right]' \text{ is finite}.
\end{align}
Next, let $m\leq n$ and assume the inductive hypothesis that for all $s\in\NN_0$ such that $s\leq m-1$,
\begin{align*}
    \lim_{x\to 1^-}\left[(1-x)^{\al+{s}}f^{(s)}(x)\right]^{(s-1)} \text{ is finite}.
\end{align*}
Then, $[f,1]_m(1)$ is finite because $f\in\cD\ti{max}^{{\bf J},n}$, so the expression of interest becomes
\begin{align*}
  [f,1]_m(1)&=\lim_{x\to 1^-}\sum_{k=1}^m\sum_{j=1}^k(-1)^{k+j}\left[(1-x)^{\al+k}f^{(k)}(x)\right]^{(k-j)}(1)^{(j-1)}\\
    &=\lim_{x\to 1^-}\sum_{k=1}^m(-1)^{k+1}\left[(1-x)^{\al+k}f^{(k)}(x)\right]^{(k-1)} .
\end{align*}
The terms $k=1,\dots,m-1$ are all finite due to our inductive hypothesis, so the term for $k=m$ must also be finite. The claim in equation \eqref{e-inductionclaim1}, and hence the base case $s=1$ in equation \eqref{e-inductionclaimgeneral}, is thus proven. 

Unfortunately, the case $s=1$ is not representative of what is to come, as the sesquilinear forms become more complicated for higher values of $s$. We will show the case $s=2$ for the sake of clarity. Examine,
\begin{align*}
    [f,(1-x)]_2(1)=\lim_{x \to 1^-}\sum_{k=1}^2\sum_{j=1}^k(-1)^{k+j}&\bigg\{\left[(1-x)^{\al+k}(1-x)^{(k)}\right]^{(k-j)}f^{(j-1)}(x) \\ &-\left[(1-x)^{\al+k}f^{(k)}(x)\right]^{(k-j)}(1-x)^{(j-1)}\bigg\}\text{ is finite}.\nonumber
\end{align*}
The terms when $k=j=1$ are finite by Proposition \ref{t-jacobimax}, and when $k=2$, $j=1$ by equation \eqref{e-inductionclaim1}. The remaining term, generated when $k=j=2$, is thus finite: 
\begin{align}\label{e-bestx}
    \lim_{x\to 1^-}(1-x)^{\al+2}f''(x).
\end{align}
Furthermore,
\begin{align*}
[f,(1-x)]_3(1)=\lim_{x \to 1^-}\sum_{k=1}^3\sum_{j=1}^k(-1)^{k+j}&\bigg\{\left[(1-x)^{\al+k}(1-x)^{(k)}\right]^{(k-j)}f^{(j-1)}(x) \\ &-\left[(1-x)^{\al+k}f^{(k)}(x)\right]^{(k-j)}(1-x)^{(j-1)}\bigg\}\text{ is finite}.\nonumber
\end{align*}
The terms where $j=k$ are finite by Proposition \ref{t-jacobimax} and equation \eqref{e-bestx}. When j=1, terms are finite by equation \eqref{e-inductionclaim1}. The remaining term, generated when $j=2$ and $k=3$,
\begin{align}
\lim_{x\to 1^-}\left[(1-x)^{\al+3}f'''(x)\right]'
\end{align}
is thus finite as well.

An induction argument on the sesquilinear form used, similar to that used to show equation \eqref{e-inductionclaim1}, can now be performed to show that for $r=2,\dots,n$
\begin{align*}
    \lim_{x\to 1^-}\left[(1-x)^{\al+r}f^{(r)}(x)\right]^{(r-2)}\text{ is finite}.
\end{align*}
A proof of the statement is suppressed for brevity, and we proceed with our central inductive argument. Assume the inductive hypothesis that for $s\leq m-1$ and $r=m-1,\dots, n$
\begin{align}\label{e-inductivehypothesisgeneral}
    \lim_{x\to 1^-}\left[(1-x)^{\al+r}f^{(r)}(x)\right]^{(r-s)}\text{ is finite}.
\end{align}

We calculate that
\begin{align}\label{e-hugejacobisum}
[f,(1-x)&^{m-1}]_{m}(1) &&\\
&=\lim_{x \to 1^-}\sum_{k=1}^{m}\sum_{j=1}^k&&{\hspace{-4mm}}(-1)^{k+j}\bigg\{\left[(1-x)^{\al+k}\left[(1-x)^{m-1}\right]^{(k)}\right]^{(k-j)}f^{(j-1)}(x)\nonumber\\
    &   &&-\left[(1-x)^{\al+k}f^{(k)}(x)\right]^{(k-j)}\left[(1-x)^{m-1}\right]^{(j-1)}\bigg\}\text{ is finite}.\nonumber
\end{align}
Terms where $j=k<m$ are finite by Proposition \ref{t-jacobimax} and the special case of the inductive hypothesis, as stated in equation \eqref{e-specialcase}. Let $k\leq m-1$. Since $j\leq k$, the terms on the left hand side of the sum, up to a finite constant, can be simplified near $x=1$ as
\begin{align}\label{e-lhsjacobi}
    \lim_{x \to 1^-}\sum_{k=1}^{m}\sum_{j=1}^k\left[(1-x)^{\al+k}\left[(1-x)^{m-1}\right]^{(k)}\right]^{(k-j)}f^{(j-1)}(x) \\
    \approx \lim_{x \to 1^-}\sum_{k=1}^{m}\sum_{j=1}^k (1-x)^{\al+(j-1)+(m-k)}f^{(j-1)}(x),\nonumber
\end{align}
all of which are finite by the inductive hypothesis. If $k=m$, all of the terms in equation \eqref{e-lhsjacobi} are 0. 

Terms on the right hand side of the sum in equation \eqref{e-hugejacobisum} are finite for $k\leq m-1$, and $j\leq k$, by the inductive hypothesis. The term where both $j=k=m$ thus finally yields
\begin{align}
    \lim_{x\to 1^-}(1-x)^{\al+m}f^{(m)}(x)\text{ is finite}.
\end{align}
We similarly calculate that 
\begin{align*}
    [f,(1-x)^{m-1}]_{m+1}(1)=\lim_{x \to 1^-}\sum_{k=1}^{m+1}&\sum_{j=1}^k(-1)^{k+j}
 \bigg\{\left[(1-x)^{\al+k}\left[(1-x)^{m-1}\right]^{(k)}\right]^{(k-j)}f^{(j-1)}(x)\\
 &-\left[(1-x)^{\al+k}f^{(k)}(x)\right]^{(k-j)}\left[(1-x)^{m-1}\right]^{(j-1)}\bigg\}\text{ is finite}.\nonumber
\end{align*}
Every term except for when $k=m+1$ and $j=m$ is finite either by the inductive hypothesis or is simply $0$ (i.e.~when $j=k=m+1$). Hence,
\begin{align*}
    \lim_{x\to 1^-}\left[(1-x)^{\al+m+1}f^{(m+1)}(x)\right]'\text{ is finite}.
\end{align*}
Again, an induction argument on the sesquilinear form used, similar to that used to show equation \eqref{e-inductionclaim1}, can now be performed to show that for $k=m,\dots,n$
\begin{align*}
    \lim_{x\to 1^-}\left[(1-x)^{\al+k}f^{(k)}(x)\right]^{(k-m)}\text{ is finite}.
\end{align*}

The claim in equation \eqref{e-inductionclaimgeneral} is thus proven by induction for $s\leq k\leq n-1$. The claim is not true for $s\leq k=n$ simply because the index of the sesquilinear form cannot be further increased to $n+1$ in the previous calculations. However, the intermediate result that 
\begin{align*}
    \lim_{x\to 1^-}(1-x)^{\al+n}f^{(n)}(x)\text{ is finite},
\end{align*}
does immediately follow. The theorem is thus proven.
\end{proof}

These properties of functions in the maximal domain also determine which functions are in the minimal domain, similar to Corollaries \ref{t-legendre+} and \ref{t-jacobi1}.

\begin{theo}\label{t-njacobiminimal}
The functions $\f_s^+$ and $\f_s^-$, $s\geq n$, belong to the minimal domain associated with the $n$th composition of the Jacobi differential expression \eqref{e-njacobi}.
\end{theo}

\begin{proof}
We will show the theorem for the functions $\f_s^+$, $s\geq n$, and the corresponding statement for the functions $\f_s^-$ will follow. 

The assertion that $\f_s^+$ belongs to the minimal domain of $\ell^n_{{\bf J}}$, for all $s\geq n$, means that for all $f\in \cD\ti{max}^{{\bf J},n}$ we have $[f,\f_s^+]_n(1)=0$. We deconstruct the expression for the sesquilinear form given by \eqref{e-njacobisesqui} into the terms
\begin{align*}
P(x):=\sum_{k=1}^n\sum_{j=1}^k(-1)^{k+j}\left[(1-x)^{\al+k}(1+x)^{\beta+k}\left[(1-x)^s\right]^{(k)}\right]^{(k-j)}f^{(j-1)}(x), \\
N(x):=\sum_{k=1}^n\sum_{j=1}^k(-1)^{k+j}\left[(1-x)^{\al+k}(1+x)^{\beta+k}f^{(k)}(x)\right]^{(k-j)}\left[(1-x)^s\right]^{(j-1)}.
\end{align*}
It is clear that the $(1+x)^{\beta+k}$ terms above can be ignored in our calculations as they contribute a factor of at most $2^{\beta+n}$ near $x=1$, which is finite. Other finite constants will also be omitted, as they can be pulled outside the limits. We set out to prove the stronger condition that
\begin{align*}
\lim_{x\to 1^-}P(x)=\lim_{x\to 1^-}N(x)=0.
\end{align*}
Next, we calculate that
\begin{align}\label{e-brokensum}
\lim_{x\to 1^-} P(x) &\approx \lim_{x\to 1^-}\sum_{k=1}^n\sum_{j=1}^k\left[(1-x)^{\al+k}(1-x)^{s-k}\right]^{(k-j)}f^{(j-1)}(x)\nonumber \\
&=\lim_{x\to 1^-}\sum_{k=1}^n\sum_{j=1}^k\left[(1-x)^{\al+s}\right]^{(k-j)}f^{(j-1)}(x)\nonumber \\
&\approx \lim_{x\to 1^-}\sum_{k=1}^n\sum_{j=1}^k(1-x)^{\al+s-k+j}f^{(j-1)}(x)\nonumber\\
&=\lim_{x\to 1^-}\sum_{k=1}^n\sum_{j=1}^k(1-x)^{s-k+1}(1-x)^{\al+j-1}f^{(j-1)}(x),\nonumber\\
&=\lim_{x\to 1^-}\sum_{k=1}^n\left[(1-x)^{s-k+1}\sum_{j=1}^k(1-x)^{\al+j-1}f^{(j-1)}(x)\right].
\end{align}
Theorem \ref{t-maxdomaindecomp} says that 
\begin{align*}
\lim_{x\to 1^-}(1-x)^{\al+j-1}f^{(j-1)}(x)\text{ is finite},
\end{align*}
for $j=2,\dots,n$. Additionally, for $k=1,\dots,n$, and $s\geq n$,  
\begin{align*}
\lim_{x\to 1^-}(1-x)^{s-k+1}=0.
\end{align*}
This implies
\begin{align*}
    \lim_{x\to 1^-}\sum_{k=1}^n\left[(1-x)^{s-k+1}\sum_{j=2}^k(1-x)^{\al+j-1}f^{(j-1)}(x)\right]=0.
\end{align*}
The remaining case where $j=1$ requires Proposition \ref{t-jacobimax} to see:
\begin{align*}
    \lim_{x\to 1^-}\sum_{k=1}^n(1-x)^{\al+s-k+1}f(x)=0.
\end{align*}
We conclude that the limit in equation \eqref{e-brokensum} is in fact 0, as desired. 

On the other hand, the limit of $N(x)$ is immediately known to be finite for each choice of $j$ and $k$ due to Theorem \ref{t-maxdomaindecomp}. Multiplication by the term $(1-x)^{s-j+1}$ then shows that the limit must be 0 for each $j$ and $k$. The theorem is thus proven.
\end{proof}

The properties of the maximal domain $\cD\ti{max}^{{\bf J},n}$ in Theorem \ref{t-maxdomaindecomp} were revealed by the interaction between the composed operator $\ell_{\bf J}^n$ and behavior near endpoints described by the functions $\left\{\f_j^-\right\}_{j=0}^{n-1}$ and $\left\{\f_j^+\right\}_{j=0}^{n-1}$. Theorem \ref{t-njacobiminimal} then says that all other associated endpoint behaviors, that occur when $j\in\NN$ and $j\geq n$, are in the minimal domain. However, the deficiency indices of $\ell_{\bf J}^n$ are $(2n,2n)$ and there are only $2n$ functions that don't seem to be in the minimal domain so far. The other $2n$ must behave differently.

We define a class of $C^{\infty}[-1,1]$ functions by their boundary asymptotics, with
elements of the class denoted by $\psi_j^+$ and $\psi_j^-$, for $j\in\NN_0$:
\begin{equation}\label{e-secondkinddefn}
\begin{aligned}
\psi_j^+(x)&:=\left.
\begin{cases}
(1-x)^{-\al+j}, & \text{ for }x \text{ near }1 \\
0, & \text{ for }x \text{ near }-1
\end{cases}
\right\}, \text{ and} \\
\psi_j^-(x)&:=\left.
\begin{cases}
0, & \text{ for }x \text{ near }1 \\
(1+x)^{-\beta+j}, & \text{ for }x \text{ near }-1
\end{cases}
\right\}.
\end{aligned}
\end{equation}
the dependence of the functions $\psi_j^+$ and $\psi_j^-$ on the parameters $\al$ and $\beta$ is suppressed here for simplicity. The choice of these functions is explained in Remark \ref{r-choices}, but the following lemma tells us that such behavior does occur in the maximal domain.

\begin{lem}\label{l-secondkindmax}
The functions $(1-x)^{-\al+j}$ and $(1+x)^{-\beta+j}$, for $j\in\NN_0$, are in the maximal domain $\cD\ti{max}^{{\bf J},n}$.
\end{lem}

\begin{proof}
We will show the lemma for the functions $(1-x)^{-\al+j}$ and the analogous result will follow for the functions $(1+x)^{-\beta+j}$. Fix $n$ and let $j\in\NN_0$. It is immediately clear that each function is in $L^2_{\al,\beta}(-1,1)$ and possesses the desired differentiability properties. It remains only to show that $\ell_{{\bf J}}^n[(1-x)^{-\al+j}]\in L^2_{\al,\beta}(-1,1)$. Calculate
\begin{align*}
\ell_{{\bf J}}^n[(1-x)^{-\al+j}]&\approx\dfrac{1}{(1-x)^{\al}(1+x)^{\beta}}\sum_{k=1}^n\left[(1-x)^{\al+k}(1+x)^{\beta+k}(1-x)^{-\al+j-k}\right]^{(k)} \\
&=\dfrac{1}{(1-x)^{\al}(1+x)^{\beta}}\sum_{k=1}^n\left[(1+x)^{\beta+k}(1-x)^j\right]^{(k)} \\
&\approx\dfrac{1}{(1-x)^{\al}(1+x)^{\beta}}\sum_{k=1}^n\sum_{i=0}^k\left[(1+x)^{\beta+k-i}(1-x)^{j-k+i}\right] \\
&=\sum_{k=1}^n\sum_{i=0}^k\left[(1+x)^{k-i}(1-x)^{-\al+j-k+i}\right].
\end{align*}
The differentiation performed on the second line requires that $j\geq k$ to ensure the term is nonzero. We focus on these values of $k$, as otherwise the term is trivially in $L^2_{\al,\beta}(-1,1)$. The last line then states that the worst behavior in any term of the double sum is $(1-x)^{-\al}$, which occurs when $j=k$ and $i=0$. This term is still in $L^2_{\al,\beta}(-1,1)$ though. Every term in the double sum is thus in $L^2_{\al,\beta}(-1,1)$ so the result follows. 
\end{proof}

The functions in equation \eqref{e-firstkind} and equation \eqref{e-secondkinddefn} indeed exhibit behavior that is not in the minimal domain. 

\begin{theo}\label{t-secondkinddefect}
The functions $\f_j^+$, $\psi_j^+$, $\f_j^-$ and $\psi_j^-$, for $j\in\NN_0$ such that $j<n$, are not in the minimal domain $\cD\ti{min}^{{\bf J},n}$.
\end{theo}

\begin{proof}
We will prove the theorem for the functions $\f_j^+$ and $\psi_j^+$; the analogous result will follow for the functions $\f_j^-$ and $\psi_j^-$. Fix $n$ and let $y\in\NN_0$ such that $y<n$. Functions in the minimal domain yield a value of $0$ when paired in the sesquilinear form with any other function from the maximal domain, so the strategy will be to show that a nonzero value is achieved when paired against one function in particular. By definition, 
\begin{align*}
    \left[\f_y^+,\psi_{n-1-y}^+\right]_n(-1)=0.
\end{align*}
Thus, consider $\left[\f_y^+,\psi_{n-1-y}^+\right]_n(1)$. As in the proof of Theorem \ref{t-njacobiminimal}, we break the sesquilinear form up into two pieces:
\begin{align*}
    P(x):=&\sum_{k=1}^n\sum_{j=1}^k(-1)^{k+j}\left[(1-x)^{\al+k}(1+x)^{\beta+k}\left[(1-x)^{-\al+n-1-y}\right]^{(k)}\right]^{(k-j)}\left[(1-x)^y\right]^{(j-1)}, \\
    N(x):=&\sum_{k=1}^n\sum_{j=1}^k(-1)^{k+j}\left[(1-x)^{\al+k}(1+x)^{\beta+k}\left[(1-x)^y\right]^{(k)}\right]^{(k-j)}\left[(1-x)^{-\al+n-1-y}\right]^{(j-1)}.
\end{align*}
This yields
\begin{align*}
    \left[\f_y^+,\psi_{n-1-y}^+\right]_n(1)=\lim_{x\to 1^-}P(x)-N(x).
\end{align*}
The explicit values of constants here do not seem to be valuable, so we don't present the intermediate steps in the following calculations. First,
\begin{align*}
    \lim_{x\to 1^-}P(x)&\approx\lim_{x\to 1^-}\sum_{k=1}^n\sum_{j=1}^k\left[(1+x)^{\beta+k}(1-x)^{n-1-y}\right]^{(k-j)}(1-x)^{y-j+1} ~\text{ for }~y\geq j-1, \\
    &\approx\lim_{x\to 1^-}\sum_{k=1}^n\sum_{j=1}^k\sum_{s=0}^{k-j}(1+x)^{\beta+k-s}(1-x)^{n+s-k} ~\text{ for }~n-1-y\geq k-j-s.
\end{align*}
If either of the conditions above are not met, then the term for such $j$ and $k$ is 0. It is also evident that the exponent $n+s-k$ will now make all other terms 0 in the limit unless $s=0$ and $k=n$. The two conditions above can thus be rewritten as $j\leq y+1$ and $j\geq y+1$. We can thus conclude that a single finite term remains, when $k=n$ and $j=y+1$. 
Second, 
\begin{align*}
    \lim_{x\to 1^-}N(x)&\approx\lim_{x\to 1^-}\sum_{k=1}^n\sum_{j=1}^k\left[(1+x)^{\beta+k}(1-x)^{\al+y}\right]^{(k-j)}(1-x)^{-\al+n-y-j} ~\text{ for }~y\geq k, \\
    &\approx\lim_{x\to 1^-}\sum_{k=1}^n\sum_{j=1}^k\sum_{t=0}^{k-j}(1+x)^{\beta+k-t}(1-x)^{n+t-k},
\end{align*}
where again if $y<k$ terms are 0. However, the only way the exponent $n+t-k=0$ is if $t=0$ and $k=n$. This contradicts our earlier assumption though, as $y\leq n-1$. We can conclude that all terms are 0 and the limit of $N(x)$ as $x\to 1^-$ is zero. In conclusion, we have shown that
\begin{align*}
    \left[\f_y^+,\psi_{n-1-y}^+\right]_n\Big|_{-1}^1= c(\al,\beta,y,n),
\end{align*}
a nonzero, finite constant depending on $\al$, $\beta$ $y$ and $n$. The theorem follows. 
\end{proof}

The functions $\left\{\f_j^{\pm}\right\}_{j=0}^{n-1}$ clearly have a special interaction with the functions $\left\{\psi_j^{\pm}\right\}_{j=0}^{n-1}$. More can be said when the indices of these functions add to be more than $n$. 

\begin{lem}\label{l-overn}
Let $s,t\in\NN_0$ such that $s,t\leq n$. If $s+t\geq n$, then
\begin{align*}
    \left[\f_s^+,\psi_t^+\right]_n\Big|_{-1}^1=\left[\f_s^-,\psi_t^-\right]_n\Big|_{-1}^1=0.
\end{align*}
\end{lem}

\begin{proof}
We will prove that
\begin{align}\label{e-overn}
    \left[\f_s^+,\psi_t^+\right]_n\Big|_{-1}^1=0,
\end{align}
and the analogous result will follow for the functions defined by their asymptotics at $x=-1$. Fix $n$ and let $s,t\in\NN_0$ such that $s+t\geq n$. Consider equation \eqref{e-overn} and notice that evaluation at the endpoint $x=-1$ is 0 by definition. Break down the sesquilinear form at the endpoint $x=1$ into two parts, as in the proof of Theorem \ref{t-secondkinddefect}:
\begin{align*}
    P(x):=&\sum_{k=1}^n\sum_{j=1}^k(-1)^{k+j}\left[(1-x)^{\al+k}(1+x)^{\beta+k}\left[(1-x)^{-\al+t}\right]^{(k)}\right]^{(k-j)}\left[(1-x)^s\right]^{(j-1)}, \\
    N(x):=&\sum_{k=1}^n\sum_{j=1}^k(-1)^{k+j}\left[(1-x)^{\al+k}(1+x)^{\beta+k}\left[(1-x)^s\right]^{(k)}\right]^{(k-j)}\left[(1-x)^{-\al+t}\right]^{(j-1)}.
\end{align*}
We calculate that
\begin{align*}
    \lim_{x\to 1^-}P(x)&\approx\lim_{x\to 1^-}\sum_{k=1}^n\sum_{j=1}^k\left[(1+x)^{\beta+k}(1-x)^t\right]^{(k-j)}(1-x)^{s-j+1} ~\text{ for }~s\geq j-1, \\
    &\approx\lim_{x\to 1^-}\sum_{k=1}^n\sum_{j=1}^k\sum_{i=0}^{k-j}(1+x)^{\beta+k-i}(1-x)^{s+t-k+i+1} ~\text{ for }~t\geq k-j+i.
\end{align*}
The exponent $s+t-k+i+1$, noting the hypothesis $s+t\geq n$ and the ranges of $k$ and $i$, has a minimum value of 1. Hence, the limit is 0 because powers of $(1-x)$ will remain in each term of the last sum. Similarly,
\begin{align*}
    \lim_{x\to 1^-}N(x)&\approx\lim_{x\to 1^-}\sum_{k=1}^n\sum_{j=1}^k\left[(1+x)^{\beta+k}(1-x)^{\al+s}\right]^{(k-j)}(1-x)^{-\al+t-j+1} ~\text{ for }~s\geq k, \\
    &\approx\lim_{x\to 1^-}\sum_{k=1}^n\sum_{j=1}^k\sum_{i=0}^{k-j}(1+x)^{\beta+k-i}(1-x)^{s+t-k+i+1}.
\end{align*}
The condition $s\geq k$ immediately implies that the limit is 0 because powers of $(1-x)$ will survive in each term of the sum. Thus, equation \eqref{e-overn} is proven and the lemma follows. 
\end{proof}

Theorem \ref{t-firstjacobiany} revealed that the functions $\f_j$ were building blocks of the Jacobi polynomials near the endpoints. Orthogonal functions, like the Jacobi polynomials, will be zero when plugged into the sesquilinear form, as it is defined by Green's formula \eqref{e-greens} as the difference of two inner products. While the functions $\f_j$ are not orthogonal, they still retain this property of eliminating each other in the sesquilinear form. 

\begin{lem}\label{l-upperleft}
Let $s,t\in\NN_0$ be such that $s,t<n$. Then
\begin{align*}
\left[\f_s^+,\f_t^+\right]_n\Big|_{-1}^1=\left[\f_s^-,\f_t^-\right]_n\Big|_{-1}^1=0.
\end{align*}
\end{lem}

\begin{proof}
We will prove that the expression 
\begin{align}\label{e-upperleft}
    \left[\f_s^+,\f_t^+\right]_n\Big|_{-1}^1=0,
\end{align}
and the analogous result will follow for the functions defined by their asymptotics at $x=-1$. Fix $n$ and let $s,t\in\NN_0$ such that $s,t<n$. Consider equation \eqref{e-upperleft} and notice that evaluation at the endpoint $x=-1$ is 0 by definition. Break down the sesquilinear form at the endpoint $x=1$ into two parts, as in the proof of Theorem \ref{t-secondkinddefect}:
\begin{align*}
    P(x):=&\sum_{k=1}^n\sum_{j=1}^k(-1)^{k+j}\left[(1-x)^{\al+k}(1+x)^{\beta+k}\left[(1-x)^t\right]^{(k)}\right]^{(k-j)}\left[(1-x)^s\right]^{(j-1)}, \\
    N(x):=&\sum_{k=1}^n\sum_{j=1}^k(-1)^{k+j}\left[(1-x)^{\al+k}(1+x)^{\beta+k}\left[(1-x)^s\right]^{(k)}\right]^{(k-j)}\left[(1-x)^t\right]^{(j-1)}.
\end{align*}
We calculate that
\begin{align*}
    \lim_{x\to 1^-}P(x)&\approx\lim_{x\to 1^-}\sum_{k=1}^n\sum_{j=1}^k\left[(1+x)^{\beta+k}(1-x)^{\al+t}\right]^{(k-j)}(1-x)^{s-j+1} ~\text{ for }~t\geq k,~s\geq j-1 \\
    &\approx\lim_{x\to 1^-}\sum_{k=1}^n\sum_{j=1}^k\sum_{i=0}^{k-j}(1+x)^{\beta+k-i}(1-x)^{\al+s+t-k+i+1}.
\end{align*}
The exponent $\al+s+t-k+i+1$, noting the hypothesis $t\geq k$ and the ranges of $s$ and $i$, has a minimum value of $\al+1$. Hence, the limit is 0 because powers of $(1-x)$ will remain in each term of the last sum. Similarly,
\begin{align*}
    \lim_{x\to 1^-}N(x)&\approx\lim_{x\to 1^-}\sum_{k=1}^n\sum_{j=1}^k\left[(1+x)^{\beta+k}(1-x)^{\al+s}\right]^{(k-j)}(1-x)^{t-j+1} ~\text{ for }~s\geq k,~t\geq j-1 \\
    &\approx\lim_{x\to 1^-}\sum_{k=1}^n\sum_{j=1}^k\sum_{i=0}^{k-j}(1+x)^{\beta+k-i}(1-x)^{\al+s+t-k+i+1}.
\end{align*}
The condition $s\geq k$ again implies that the limit is 0 because powers of $(1-x)$ will survive in each term of the sum. Thus, equation \eqref{e-upperleft} is proven and the lemma follows. 
\end{proof}

It does not appear that the functions $\psi_j$ can be used in the sesquilinear form, in some order, to discern properties of the maximal domain similar to Theorem \ref{t-maxdomaindecomp}. It is also unclear what happens when two $\psi_j$ functions are paired in the sesquilinear form. However, the interactions discovered so far are strong enough to have important repercussions for self-adjoint extensions of the operator $\ell_{\bf J}^n$.

%%%%%%%%%%%%%%%%%%%%%%%%%%%%%
%%%%%%%%%%%%%%%%%%%%%%%%%%%%%
\subsection{Self-Adjoint Extensions}\label{ss-nextensions}
%%%%%%%%%%%%%%%%%%%%%%%%%%%%%
%%%%%%%%%%%%%%%%%%%%%%%%%%%%%

We now use the set of functions that we have found to be in $\cD\ti{max}^{{\bf J},n}\backslash\cD\ti{min}^{{\bf J},n}$ at each endpoint to define two finite-dimensional subspaces of $\cD\ti{max}^{{\bf J},n}$:

\begin{align*} 
{\bf D}^n_-:=\spa\left\{\left\{\f_j^-\right\}_{j=0}^{n-1},\left\{\psi_j^-\right\}_{j=0}^{n-1}\right\}, ~~{\bf D}^n_+:=\spa\left\{\left\{\f_j^+\right\}_{j=0}^{n-1},\left\{\psi_j^+\right\}_{j=0}^{n-1}\right\}.
\end{align*}

We first show that the basis functions in each of the two spaces are linearly independent modulo the minimal domain $\cD\ti{min}^{{\bf J},n}$. As the two subspaces themselves are clearly linearly independent, each basis function being 0 at the other endpoint, this will result in showing that 
\begin{align*}
    \dim\left({\bf D}^n_-\dotplus{\bf D}^n_+\right)=4n.
    \end{align*}

\begin{theo}\label{t-basislinind}
The basis functions of ${\bf D}^n_-$ and ${\bf D}^n_+$ are all linearly independent modulo the minimal domain $\cD\ti{min}^{{\bf J},n}$.
\end{theo}

\begin{proof}
We will show the result for the basis functions of ${\bf D}^n_+$, and the proof for the basis functions of ${\bf D}^n_-$ will follow analogously. Fix $n$ and consider the matrix of sesquilinear forms\footnote{The ``+'' superscript is suppressed for all functions in the matrix for the sake of simplicity.}, similar to that in \cite[Equation 3.1]{FFL}, defined by
\begin{align}\label{e-M}
\sbox0{$\begin{matrix}[\f_0,\f_0] & \dots & [\f_0,\f_{n-1}] \\ \vdots & \ddots & \vdots \\ [\f_{n-1},\f_0] & \dots & [\f_{n-1},\f_{n-1}]\end{matrix}$}
\sbox1{$\begin{matrix}[\f_0,\psi_0] & \dots & [\f_0,\psi_{n-1}] \\ \vdots & \ddots & \vdots \\ [\f_{n-1},\psi_0] & \dots & [\f_{n-1},\psi_{n-1}]\end{matrix}$}
\sbox2{$\begin{matrix}[\psi_0,\f_0] & \dots & [\psi_0,\f_{n-1}] \\ \vdots & \ddots & \vdots \\ [\psi_{n-1},\f_0] & \dots & [\psi_{n-1},\f_{n-1}]\end{matrix}$}
\sbox3{$\begin{matrix}[\psi_0,\psi_0] & \dots & [\psi_0,\psi_{n-1}] \\ \vdots & \ddots & \vdots \\ [\psi_{n-1},\psi_0] & \dots & [\psi_{n-1},\psi_{n-1}]\end{matrix}$}
{\bf M}_n:=\left(
\begin{array}{c|c}
\usebox{0}&\usebox{1}\vspace{-.35cm}\\\\
\hline\vspace{-.35cm}\\
 \vphantom{\usebox{0}}\usebox{2}&\usebox{3}
\end{array}
\right),
\end{align}
where each sesquilinear form is evaluated from $-1$ to $1$. If the matrix ${\bf M}_n$ has full rank, then the vectors $\left\{\f_j^+\right\}_{j=0}^{n-1}$ and $\left\{\psi_j^+\right\}_{j=0}^{n-1}$ are all linearly independent modulo the minimal domain by \cite[Prop. 2.14]{FFL}. Lemma \ref{l-overn} says that the upper-right and lower-left blocks of ${\bf M}_n$ are upper diagonal matrices, and Theorem \ref{t-secondkinddefect} says that those diagonal entries are nonzero. Lemma \ref{l-upperleft} additionally says that the upper-left block consists of only zeros. We conclude that ${\bf M}_n$ has full rank and the result follows. 
\end{proof}

Theorem \ref{t-secondkinddefect} shows that the $\f_j$ and $\psi_j$ functions are not in the minimal domain associated with powers of the Jacobi differential equation, but are in the maximal domain. Theorem \ref{t-basislinind} then says that these functions span a $4n$ dimensional subspace of the maximal domain. Thus, we can appeal to equation \eqref{e-vN}, which says that there are only $4n$ dimensions in the maximal domain modulo the minimal domain and draw additional conclusions.

\begin{cor}\label{c-defectspacesbasis}
The defect spaces $\cD_+^{{\bf J}, n}\dotplus\cD_-^{{\bf J}, n}={\bf D}^n_-\dotplus{\bf D}^n_+$.
\end{cor}

\begin{proof}
Fix $n\in\NN$. Theorem \ref{t-basislinind} says that ${\bf D}^n_-\dotplus{\bf D}^n_+$ is comprised of $4n$ dimensions in the maximal domain $\cD\ti{max}^{{\bf J},n}$ but not in the minimal domain $\cD\ti{min}^{{\bf J},n}$. Hence, ${\bf D}^n_-\dotplus{\bf D}^n_+\subset\cD_+^{{\bf J}, n}\dotplus\cD_-^{{\bf J}, n}$. Equation \ref{e-vN} and the fact that the dimensions of both spaces coincide then implies that the defect spaces $\cD_+^{{\bf J}, n}\dotplus\cD_-^{{\bf J}, n}$ are equal to ${\bf D}^n_-\dotplus{\bf D}^n_+$.
\end{proof}

Note that Theorem \ref{t-njacobiminimal} also now follows from Corollary \ref{c-defectspacesbasis}.

\begin{cor}\label{t-jacobilinind}
The collection $\left\{P^{(\al,\beta)}_m\right\}_{m=0}^{n-1}$ is linearly independent modulo the minimal domain.
\end{cor}

\begin{proof}
The result is immediately implied by the decomposition of Jacobi polynomials in equation \eqref{e-jacobihyper} and Theorem \ref{t-basislinind}.
\end{proof}

Self-adjoint extension theory is centered around the defect spaces, so their decomposition into the span of simple functions offers a convenient interpretation of Glazman--Krein--Naimark Theory via Theorem \ref{t-gknmatrix}.

\begin{cor}\label{t-finallyamatrix}
All self-adjoint extensions of the symmetric differential expression $\ell^n\ti{{\bf J}}$ on $\cD\ti{min}^{{\bf J},n}$ can be expressed as an $2n\times 2n$ unitary matrix $U_n:{\bf D}_+^n\to {\bf D}_-^n$.
\end{cor}

\begin{proof}
Every self-adjoint extension ${\bf L}=\left\{\ell^n\ti{{\bf J}},\cD_{{\bf L}}\right\}$ of the operator ${\bf L}\ti{min}=\left\{\ell^n\ti{{\bf J}},\cD\ti{min}^{{\bf J},n}\right\}$ can be characterized by means of a unitary $2n\times 2n$ matrix $V(2n):\cD_+^{{\bf J}, n}\to \cD_-^{{\bf J}, n}$ via Theorem \ref{t-gknmatrix}. Theorem \ref{t-basislinind} then says that ${\bf D}^n_-$ and ${\bf D}^n_+$ are each $2n$ subspaces of $\cD^{{\bf J},n}\ti{max}\backslash\cD^{{\bf J},n}\ti{min}$.

Equation \eqref{e-vN} and the fact that the dimensions of both spaces coincide allows us to define change of basis maps as the unitary matrices $U_+:\cD_+^{{\bf J}, n}\to {\bf D}^n_+$ and $U_-:\cD_-^{{\bf J}, n}\to {\bf D}^n_-$. The following commutative diagram holds:
\[ \begin{tikzcd}
\cD_+^{{\bf J}, n} \arrow{r}{V(n)} \arrow[swap]{d}{U_+} & \cD_-^{{\bf J}, n} \arrow{d}{U_-} \\%
{\bf D}^n_+ \arrow{r}{U(n)}& {\bf D}^n_-
\end{tikzcd}
\]
Therefore, $U(n)=U_-V(n)U_+^{-1}$ and is unitary.
\end{proof}

The structure of the self-adjoint extensions revealed in the previous Corollaries has several implications. Indeed, it has already been mentioned in equation \eqref{e-maxnesting} that the maximal domain associated with a power $p$ will be contained in the maximal domain associated with a power $q$ if $p<q$. Is there any similar assertion we can make for minimal domains or defect spaces? 

A priori, the answer is no. Definition \ref{d-defect} says that the defect spaces consist of solutions to the eigenvalue problem for one choice of $\lambda$ in the upper-half plane and one choice in the lower-half plane, although here it was given with $\lambda=\pm i$. Ordinary differential equations with two limit-circle endpoints, like Jacobi, have discrete spectrum though and solutions to the uncomposed equation are known. Let the two linearly independent solutions (called deficiency elements) be denoted by $\hat{\f}_{\la}$ and $\hat{\psi}_{\la}$, so that $\ell_{\al,\beta}[\hat{\f_{\la}}]=\la \f_{\la}$ and $\ell_{\al,\beta}[\hat{\psi_{\la}}]=\la \psi_{\la}$. Consider the change of deficiency elements between $n=2$ and $n=3$ when the defect spaces are set by $\la=\pm i$. They will have indices that are roots of unity of different orders but share no other obvious properties. 

Corollary \ref{c-defectspacesbasis} does not say that defect spaces are nested within each other like equation \eqref{e-maxnesting} because they are associated with different maximal domains. It does say that defect spaces change in a predictable, easily described way when $n$ is increased though: the possibilities for endpoint behavior increases, as $\f_j$ and $\psi_j$ are no longer in the minimal domain for larger $j$. 

\begin{rem}\label{r-choices}
The endpoint behavior exhibited by the functions $\f_j^+$, $\psi_j^+$, $\f_j^-$ and $\psi_j^-$, for $j\in\NN_0$, is chosen to approximate that of deficiency elements. To demonstrate, consider the two linearly independent solutions to the equation $\ell_{\al,\beta}[f]=\la f$. Near the endpoint $x=1$ \cite[Section 4.2]{Sz} they can be written as
\begin{align*}
    &\,_2F_1(-\mu,\mu+\al+\beta+1;\al+1;(1-x)/2), \\
   ((1-x)/2)^{-\al}&\,_2F_1(-\mu-\al,\mu+\beta+1;1-\al;(1-x)/2),
\end{align*}
and near $x=-1$ as
\begin{align*}
    &\,_2F_1(-\mu,\mu+\al+\beta+1;\beta+1;(1+x)/2), \\
   ((1+x)/2)^{-\beta}&\,_2F_1(-\mu-\beta,\mu+\al+1;1-\beta;(1+x)/2),
\end{align*}
where $\la=\mu(\mu+\al+\beta+1)$ and $\,_2F_1(a,b;c;x)$ denotes the Gauss hypergeometric series. Deficiency elements are simply these functions for specific choices of $\la$, as described above. If these series are all truncated so that only powers less than $n$ remain, the resulting behavior of $x$ is captured by linear combinations of the functions $\f_j^+$, $\psi_j^+$, $\f_j^-$ and $\psi_j^-$. Corollaries \ref{c-defectspacesbasis} and \ref{t-finallyamatrix} thus say that the only behavior of the deficiency elements contributing to the domain are those that come from our functions with $j<n$. 
\end{rem}

This structure can be used to describe domains of self-adjoint domains in different ways. Consider the self-adjoint operator ${\bf J}^n_{(\al,\beta)}$ acting via \eqref{e-njacobi} with domain that contains all of the Jacobi polynomials, $P_m^{(\al,\beta)}$. The domain of this operator can be written as
\begin{align*}
  \cD_{
  {\bf J}_{(\al,\beta)}^n}=\left\{f\in\cD\ti{max}^{(\al,\beta)}~:~\left[f,P^{(\al,\beta)}_j\right]_n(\pm 1)=0 \text{ for }j=0,\dots, n-1
  \right\}.
\end{align*}

This characterization is a consequence of Corollary \ref{t-jacobilinind} and the matrix constructions of \cite[Section 3]{FFL} which show the domain includes all of the desired functions. An in depth discussion of similar domains is included in Section \ref{s-equivalence}, where it is also shown that $\cD_{{\bf J}^n_{(\al,\beta)}}$ is the left-definite domain associated with the $n$th composition of the Jacobi differential equation.

\begin{lem}[{\cite[Definition 5.1]{EKLWY}}]\label{t-njacobiprop}
The left-definite domain associated with the $n$th composition of the Jacobi differential equation can be characterized as follows:
\begin{align*}
\cD_{{\bf J}_{(\al,\beta)}^n}=\left\{f\in\cD\ti{max}^{(\al,\beta)}~:~(1-x)^{j/2}(1+x)^{j/2}f^{(j)}\in L^2_{\al,\beta}(-1,1)\text{ for }j=0,\dots,2n\right\}.
\end{align*}
\end{lem}
 
Lemma \ref{t-njacobiprop} offers a comparison between the maximal domain, described in Theorem \ref{t-maxdomaindecomp}, and the left-definite domain. In other words, functions in the maximal domain are roughly {\em half} as regular as functions in the left-definite domain. This is the analog of the comparison made between Lemma \ref{t-stronglimitpoint} and Proposition \ref{t-jacobimax} for the ($n=1$) Jacobi operator.

The functions $\f_j^+$ and $\f_j^-$ are motivated by Jacobi polynomials, and it would be ideal to use Jacobi polynomials as boundary conditions. Therefore, it is necessary to compare the effect of using Jacobi polynomials and $\f_j^+$ and $\f_j^-$ as GKN boundary conditions.

\begin{prop}\label{t-stackingprop}
Let $m\in\NN_0$ and $f\in\cD^{{\bf J}, n}\ti{max}$. Assume $\left[f,P_j^{(\al,\beta)}\right]_n(\pm 1)=0$ for all $j=0,\dots,m$. Then $[f,(1-x)^m]_n(\pm 1)=[f,(1+x)^m]_n(\pm 1)=0$. 
\end{prop}

\begin{proof}
If $m\geq n$, the conclusion is trivially implied by Theorem \ref{t-njacobiminimal}, so assume that $m<n$. Proceed by induction. The base case for $m=0$ is trivial:
\begin{align*}
     \left[f, P^{(\al,\beta)}_0\right]_n(\pm 1)=[f,1]_n(\pm 1)=[f,(1-x)^0]_n(\pm 1)=[f,(1+x)^0]_n(\pm 1)=0.   
\end{align*}
Assume the claim is true for $m=r-1$, and let $\left[f,P_j^{(\al,\beta)}\right]_n\bigg|_{-1}^1=0$ for all $j=0,\dots,r$. We will show $[f,(1-x)^r]_n(1)=0$ and the $[f,(1+x)^r]_n(-1)=0$ statement will follow analogously. The decomposition in equation \eqref{e-jacobihyper} still applies:
\begin{align*}
    0=\left[f,P_r^{(\al,\beta)}\right]_n(1)&=\sum_{i=0}^r a_i\left[f,(1-x)^i\right]_n(1) \\
&=a_r[f,(1-x)^r]_n(1)+a_{r-1}[f,(1-x)^{r-1}]_n(1)+\dots+a_0[f,1]_n(1) \\
&=[f,(1-x)^r]_n(1).
\end{align*}
The final equality follows by the inductive hypothesis, and the proposition is proven. 
\end{proof}

The extension of Theorem \ref{t-firstjacobiany} to higher order compositions of the Jacobi differential expression is now available. Let $\cM=\{m_1,\dots,m_n\}\subset\NN_0$, denote a set of $n$ indices. Define the domain 
\begin{align}
\cD_{{\bf J}^n_{(\al,\beta)}(\cM)}=\left\{ f\in\cD\ti{max}^{{\bf J},n}~:~\left[f,P\ci{m_j}^{(\al,\beta)}\right]_n(\pm 1)=0 \text{ for all }m_j\in\cM \right\},
\end{align}
where $P\ci{m_j}^{(\al,\beta)}$ is the $m_j$-th Jacobi polynomial. 

\begin{theo}\label{t-jacobifirstton}
The self-adjoint operators 
\begin{align*}
{{\bf J}^n_{(\al,\beta)}}=\left\{\ell^n_{\bf J},\cD_{{\bf J}^n_{(\al,\beta)}} \right\} 
\text{ and } 
{{\bf J}^n_{(\al,\beta)}(\cM)}=\left\{\ell^n_{\bf J},\cD_{{\bf J}^n_{(\al,\beta)}(\cM)} \right\} 
\end{align*}
are equal.
\end{theo}

\begin{proof}
The definition in equation \eqref{e-jacobihyper} says $P_m^{(\al,\beta)}(x)$, $m\in\NN_0$, can be decomposed as
\begin{align*}
P_m^{(\al,\beta)}(x)=\sum_{j=0}^m a_j(1-x)^j
\end{align*}
for some constants $a_j\in\RR$. Let $m_i\in\cM$ and $f\in\cD_{{\bf J}^n_{(\al,\beta)}}$. If $m_i<n$ the definition of $\cD_{{\bf J}^n_{(\al,\beta)}}$ immediately implies 
\begin{align*}
    \left[f,P_{m_i}^{(\al,\beta)}\right]_n\bigg|_{-1}^1=0.
\end{align*}
Let $m_i\geq n$. Then,
\begin{align}\label{e-njacobipolysum+}
\left[f,P_{m_i}^{(\al,\beta)}\right]_n&(1)=\sum_{j=0}^m a_j\left[f,(1-x)^j\right]_n(1) \\
&=a_{n-1}\left[f,(1-x)^{n-1}\right]_n(1)+\dots +a_0[f,1]_n(1).\nonumber
\end{align}
Theorem \ref{t-njacobiminimal}, which says that $(1-x)^j$, $j\geq n$ is in the minimal domain, eliminate most terms. The $n$ remaining terms evaluate to zero, as $\cD_{{\bf J}^n_{(\al,\beta)}}$ has $n$ relevant boundary conditions via Proposition \ref{t-stackingprop}. As $m_i\in\cM$ was arbitrary, we conclude
\begin{align*}
    \left[f,P_{m_i}^{(\al,\beta)}\right]_n(1)=0,
\end{align*}
for all $m_i\in\cM$. The endpoint $-1$ follows analogously and $f\in\cD_{{\bf J}^n_{(\al,\beta)}(\cM)}$.

Now, assume that $f\in\cD_{{\bf J}^n_{(\al,\beta)}(\cM)}$ does not imply $f\in\cD_{{\bf J}^n_{(\al,\beta)}}$. The presence of $2n$ GKN boundary conditions in the definition of $\cD_{{\bf J}^n_{(\al,\beta)}(\cM)}$ means that the minimal domain is extended by a maximum of $2n$ dimensions. This maximum is achieved if and only if the boundary condition functions are all linearly independent modulo the minimal domain by the GKN2 Theorem \ref{t-gkn2}. However, the inclusion $\cD_{{\bf J}^n_{(\al,\beta)}}\subset\cD_{{\bf J}^n_{(\al,\beta)}(\cM)}$ implies that the minimal domain has been extended by at least $2n$ dimensions by Corollary \ref{t-jacobilinind}. We conclude that the domains must be equal, and hence the operators ${{\bf J}^n_{(\al,\beta)}}$ and ${{\bf J}^n_{(\al,\beta)}(\cM)}$ are equal.
\end{proof}

\begin{rem}\label{r-laguerre}
All results from this section should be easily adapted to the the left-definite spaces associated with the Laguerre operator with parameter $\al> -1$ and $\al^2\neq 1/2$. The Laguerre polynomials lie in $(0,\infty)$ and so their decomposition, in analogy to equation \eqref{e-jacobihyper}, is with respect to the monomials. The main difference is that the Laguerre differential expression normally has deficiency indices $(n,n)$. Thus, we write $n$ GKN boundary conditions and the condition of each sesquilinear form at infinity will be trivially satisfied. The functions $x^j_+$ and $x^{-\al+j}_+$ (the $+$ denoting that the monomial behavior is only near the endpoint 0) should belong to the minimal domain for $j\in\NN_0$ such that $j\geq n$. More information on powers of the Laguerre operator can be found in \cite[Section 12]{LW02}.
\end{rem}

%%%%%%%%%%%%%%%%%%%%%%%%%%%%%
%%%%%%%%%%%%%%%%%%%%%%%%%%%%%
\section{Equality of Left-Definite Domains}\label{s-equivalence}
%%%%%%%%%%%%%%%%%%%%%%%%%%%%%
%%%%%%%%%%%%%%%%%%%%%%%%%%%%%

Left-definite domains have appeared in Sections \ref{s-mindomains} and \ref{s-powers}, but mainly as a comparison to the characterizations of the maximal domain, see Lemmas \ref{t-AC}, \ref{t-stronglimitpoint} and \ref{t-njacobiprop}. However, if a Sturm--Liouville operator contains a complete set of eigenfunctions, Theorem \ref{t-leftdefortho} says that a simple identifying feature of left-definite domains associated with powers of the Sturm--Liouville operator is the inclusion of this set of eigenfunctions. Thus, left-definite domains are often the {\em nicest} self-adjoint extensions to work with, and are of particular importance.

We now let ${\bf L}^n$ be a self-adjoint operator defined by left-definite theory on $L^2[(a,b),w]$ with domain $\cD_{\bf L}^n$ that includes a complete system of orthogonal eigenfunctions. Enumerate the orthogonal eigenfunctions as $\{P_k\}_{k=0}^{\infty}$. Let ${\bf L}^n$ operate on its domain via $\ell^n[\fdot]$, a differential operator of order $2n$, $n\in\NN$, generated by composing a Sturm--Liouville differential operator with itself $n$ times. Furthermore, let ${\bf L}^n$ be an extension of the minimal operator ${\bf L}^n\ti{min}$ that has deficiency indices $(n,n)$, and the associated maximal domain be denoted by $\cD\ti{max}^n$.

Also, let $\cM=\{m_1,\dots,m_n\}\subset\NN_0$ denote a set of $n$ indices. Define
\begin{align*}
    \widetilde{\cC}_n(\cM):=\left\{f\in \cD\ti{max}^n~:~\left[f,P\ci{m_j}\right]_n\bigg|_{-1}^1=0 \text{ for all }m_j\in\cM \right\}.
\end{align*}
This allows us to compare several different possible characterizations of the left-definite domain:
\begin{align*} 
\cA_n&:=\left\{f\in \cD\ti{max}^n~:~f,f',\dots,f^{(2n-1)}\in AC\ti{loc}(a,b);
(p(x))^n f^{(2n)}\in L^2[(a,b),w]\right\},
\\
\cB_n&:=\left\{f\in \cD\ti{max}^n~:~
[f,P_j]_n\Big|_a^b=0 \text{ for }j=0,1,\dots,n-1\right\}, 
\\
\widetilde{\cC}_n&:=\bigcup_{\cM}\widetilde{\cC}_n(\cM)=\left\{f\in \cD\ti{max}^n~:~
[f,P_j]_n\Big|_a^b=0 \text{ for any }n \text{ distinct }j\in\NN \right\}, \text{ and}
\\
\cF_n&:=\left\{f\in \cD\ti{max}^n~:~\left[a_j(x)
f^{(j)}(x)\right]^{(j-1)}\Big|_a^b=0 \text{ for }j=1,2,\dots,n
\right\}.
\end{align*}

The $p(x)$ above is from the standard definition of a Sturm-Liouville differential operator, given in equation \eqref{d-sturmop}, and the $a_j(x)$'s are from the Lagrangian symmetric form of the operator in \eqref{e-lagrangian}. The union in the definition of $\widetilde{\cC}_n$ is taken over all sets of indices $\cM$ that are of size $n$, while the tilde notation is used to distinguish the domain from that of $\cC_n$ in \cite[Section 6]{FFL}. In the current notation, $\cC_n=\widetilde{\cC}_n(\cM)$ for a single unknown $\cM$. Hence, the inclusion $\cB_n\subset\widetilde{\cC}_n$ is not obvious, and would imply that if the first $n$ eigenfunctions can be used as boundary conditions for a domain, then any $n$ would suffice. The opposite inclusion $\widetilde{\cC}_n\subset\cB_n$ is thus also not trivial. 

These domains seem very different, yet progress has already been made on the equality of these domains in this manuscript and elsewhere. There is a proof of $\cA_n\subseteq \cF_n$ in \cite{LW15, LWOG} for the special case where the differential operator $\ell^n[\fdot]$ denotes the $n^{th}$ composite power of the Legendre differential expression and the eigenfunctions $\{P_k\}_{k=0}^{\infty}$ are the Legendre polynomials. This scenario for small $n$ was discussed in \cite[Section 4]{FFL}. The conditions in $\cF_n$ are particularly notable, as they represent easily testable conditions that are not in the GKN format. 

The $n$th left-definite domain, $\cD_{\bf L}^n$ is found to be equal to $\cA_n$ for the Jacobi differential operator in {\cite[Corollary 5.1]{EKLWY}} and the Laguerre differential operator in {\cite[Corollary 12.9]{LW02}}. Indeed, the main condition of $\cA_n$ simply involves the term associated with $f^{(2n)}$ when $\ell^n[f]$ is decomposed into a sum of derivatives of $f$. We restrict our attention to these operators so the equality $\cD_{\bf L}^n=\cA_n$ holds. Significant progress on the equality of the domains was achieved in \cite[Section 6]{FFL}, where the following theorem was proven.

\begin{theo}\label{t-oldtheo}
Let ${\bf L}^n$ be a self-adjoint operator defined by left-definite theory on $L^2[(a,b),w]$ with the left-definite domain $\cD_{\bf L}^n$. Let ${\bf L}^n$ operate on its domain via $\ell^n[\fdot]$, a classical Jacobi or Laguerre differential expression, with $\al,\beta>-1$ or $\al>-1$ respectively, of order $2n$, where $n\in\NN$, generated by composing the Sturm--Liouville operator with itself $n$ times. Furthermore, let ${\bf L}^n$ be an extension of the minimal operator ${\bf L}^n\ti{min}$, which has deficiency indices $(n,n)$. Assume $\cA_n=\cB_n$ and that $f\in \cF_n$ implies that $f'',\dots,f^{(2n-2)}\in L^2[(a,b),dx]$. Then $\cD_{\bf L}^n=\cA_n=\cB_n=\widetilde{\cC}_n(\cM)=\cF_n$, $\forall n\in\NN$.
\end{theo}

The decompositions, techniques and results in this paper offer significant improvements to the theorem. However, in some cases the methods from \cite{FFL} are sufficient, so we avoid reproving these facts for the sake of brevity. The improved theorem, and confirmation of \cite[Conjecture 6.1]{FFL} for the Jacobi differential equation now follows.

\begin{theo}\label{t-newtheo}
Let ${\bf L}^n$, $n\in\NN$, be a self-adjoint operator defined by left-definite theory on $L^2[(a,b),w]$ with the left-definite domain $\cD_{\bf L}^n$. Let ${\bf L}^n$ operate on its domain via $\ell^n[\fdot]$, a classical Jacobi differential expression of order $2n$, with parameters $\al,\beta>0$, generated by composing the Sturm--Liouville operator with itself $n$ times. Furthermore, let ${\bf L}^n$ be an extension of the minimal operator ${\bf L}^n\ti{min}$, which has deficiency indices $(2n,2n)$. Then $\cD_{\bf L}^n=\cA_n=\cB_n=\widetilde{\cC}_n=\cF_n$, $\forall n\in\NN$.
\end{theo}

\begin{proof}
The comments prior to Theorem \ref{t-oldtheo} mention that $\cD_{{\bf L}}^n=\cA_n$ for the Jacobi left-definite domains. The equality $\cA_n=\cB_n$ was actually shown in Corollary \ref{t-jacobilinind}, due to the framework from \cite[Theorem 3.2]{FFL}. In particular, the fact that $B_n$ includes all of the orthogonal eigenfunctions and satisfies the Glazman symmetry conditions of Theorem \ref{t-gkn1} is guaranteed by the matrix construction of \cite[Section 3]{FFL}. The equality $\cB_n=\widetilde{\cC}_n$ is already proven in Theorem \ref{t-jacobifirstton} as $\cM$ was arbitrary. Finally, $\cB_n\subset\cF_n$ was proven in \cite[Theorem 6.2]{FFL}. The reverse inclusion, $\cF_n\subset\cB_n$ then follows immediately by a dimension argument similar to that at the end of the proof of Theorem \ref{t-jacobifirstton}. The distinction between boundary conditions with limits and GKN boundary conditions is unimportant to self-adjoint extension theory, only that the minimal domain can be extended by a maximum number of dimensions which is equal to the number of boundary conditions imposed.  
\end{proof}
    
It should be mentioned that alternative proofs of some inclusions between the spaces are readily available. For instance, if $\cA_n=\cF_n$, it is possible to show $\cF_n\subset\cB_n$ using the induction arguments from Theorem \ref{t-maxdomaindecomp} in an opposite fashion. The Theorem also answers a conjecture concerning left-definite domains of the Legendre operator from \cite{LWOG}.

The main advancements of Theorem \ref{t-newtheo} are the expansion of $\widetilde{\cC}_n(\cM)$ to the larger class $\widetilde{\cC}_n$, and the removal of $L^2$ restrictions on functions from $\cF_n$. Thus, explicit boundary conditions of both GKN and the non-GKN variety have been shown to define the left-definite domains associated with powers of some classical Sturm--Liouville differential operators. 

We conclude our discussion of left-definite domains by mentioning a peculiar perspective on the domain $\cF_n$. In particular, it is easy to see upon examination that 
\begin{align*}
    \cF_n=\left\{f\in\cD\ti{max}^n ~:~ [f,1]_j\Big|_a^b=0 \text{ for } j=1,\dots,n\right\}.
\end{align*}
Indeed, the proof of Theorem \ref{t-maxdomaindecomp} analyzes such limits and evaluations of the sesquilinear form in different spaces. Such a representation is clearly not covered in the classical texts on GKN theory \cite{AG,N} and hence represent non-GKN boundary conditions, despite this description using sesquilinear forms. It is thus the only domain of the four considered that explicitly builds upon itself as the number of compositions is increased. As a comparison, $\cB_n$ reuses functions when $n$ is increased to $n+1$, but the calculations required to verify the conditions are quite different because they use the $n+1$ sesquilinear form. In the $\cF$ domains, the boundary conditions from $\cF_n$ are all included verbatim in $\cF_{n+1}$. It is also clear that the function $1$ is unique in its ability to define such domains, as no other eigenfunction is available when $n=1$ by Corollary \ref{t-jacobi1}. However, if other operators are considered, such a structure could possibly be generated by other functions.

%%%%%%%%%%%%%%%%%%%%%%%%%%%%%
%%%%%%%%%%%%%%%%%%%%%%%%%%%%%
\subsection*{Acknowledgements}
%%%%%%%%%%%%%%%%%%%%%%%%%%%%%
%%%%%%%%%%%%%%%%%%%%%%%%%%%%%

The authors would like to thank Lance Littlejohn for useful discussions about domains of the Legendre and Jacobi operators.

%%%%%%%%%%%%%%%%%%%%%%%%%%%%%
%%%%%%%%%%%%%%%%%%%%%%%%%%%%%

\end{document}